\documentclass[amstex,12pt,russian,amssymb]{article}

\usepackage{mathtext}
\usepackage[cp1251]{inputenc}
\usepackage[T2A]{fontenc}
\usepackage[russian]{babel}
\usepackage[dvips]{graphicx}
\usepackage{amsmath}
\usepackage{amssymb}
\usepackage{amsxtra}
\usepackage{latexsym}
\usepackage{ifthen}

\begin{document}

\newcounter{lemma}
\newcommand{\lemma}{\par \refstepcounter{lemma}%
{\bf Лемма \arabic{lemma}.}}

\newcounter{corollary}
\newcommand{\corollary}{\par \refstepcounter{corollary}%
{\bf Следствие \arabic{corollary}.}}

\newcounter{remark}
\newcommand{\remark}{\par \refstepcounter{remark}%
{\bf Замечание \arabic{remark}.}}

\newcounter{theorem}
\newcommand{\theorem}{\par \refstepcounter{theorem}%
{\bf Теорема \arabic{theorem}.}}

\newcounter{proposition}
\newcommand{\proposition}{\par \refstepcounter{proposition}%
{\bf Предложение \arabic{proposition}.}}

\newcounter{example}
\newcommand{\example}{\par \refstepcounter{example}%
{\bf Пример~\arabic{example}.}}

\renewcommand{\refname}{\centerline{\bf Список литературы}}

\newcommand{\proof}{{\it Доказательство.\,\,}}

\noindent УДК 517.5

{\bf Е.А.~Севостьянов} (Житомирский государственный университет
имени Ивана Франко, Институт прикладной математики и механики НАН
Украины, г.~Славянск)

\medskip
{\bf Є.О.~Севостьянов} (Житомирський державний університет імені
Івана Фран\-ка, Інститут прикладної математики і механіки НАН
України, м.~Слов'янськ)

\medskip
{\bf E.A.~Sevost'yanov} (Zhytomyr Ivan Franko State University,
Institute of Applied Ma\-the\-ma\-tics and Mechanics of NAS of
Ukraine, Slov'yans'k)

\medskip
{\bf О классах Орлича-Соболева на фактор-пространствах}

\medskip
{\bf Про класи Орліча-Соболєва на фактор-просторах}

\medskip
{\bf On Orlicz-Sobolev classes on factor spaces}

\medskip\medskip
Изучаются фактор-пространства единичного шара размерности, не
меньшей трёх, по некоторой группе мёбиусовых преобразований. Для
отображений таких пространств получена оценка искажения модуля
семейств сфер. В качестве приложений, получены теоремы о локально и
граничном поведении классов Орлича--Соболева фактор-пространств.

\medskip\medskip
Досліджено фактор-простори одиничної кулі розмірності, не меншої
трьох, по деякій групі мебіусових перетворень. Для відображень таких
просторів отримано оцінку спотворення модуля сімей сфер. Як
застосування, отримано теореми про локальну і межову поведінку
класів Орліча--Соболєва двох фактор-просторів.

\medskip\medskip
We study the factor-spaces of the unit ball of dimension, not less
than three, by a certain group of M\"{o}bius transformations. For
mappings of such spaces, an estimate of the distortion of the
modulus of families of spheres is obtained. As applications, we
obtain theorems on the locally and boundary behavior of the
Orlicz--Sobolev classes between factor spaces.

\newpage
{\bf 1. Введение.} Настоящая статья посвящена исследованию
отображений с ограниченным и конечным искажением, а также классов
Соболева и Орлича--Соболева, активно изучаемых последнее время (см.,
напр., \cite{GRY}, \cite{IS}, \cite{KRSS} и \cite{MRSY}). %Как правило, мы
%рассматриваем отображения, принадлежащие классу Орлича-Соболева при
%условии Кальдерона на определяющую его функцию (см.~\cite{Cal}).
Ниже будет рассмотрен случай, когда отображения заданы в области
некоторого фактор-пространства по группе конформных автоморфизмов
единичного шара и действуют в область такого же вида. Отметим, что
изучению квазирегулярных отображений фактор-пространств посвящены
публикации~\cite{MS$_1$} и \cite{MS$_2$}. Также отметим несколько
недавних работ по римановым поверхностям,
см.~\cite{RV}--\cite{RV$_1$}, в которых рассматривались вопросы
локального и граничного поведения отображений. В силу теоремы
Пуанкаре об униформизации, фактор-пространства, изучаемые в данной
статье, можно рассматривать как наиболее близкий многомерный аналог
римановой поверхности. Основные результаты относятся к
ситуации~$n\geqslant 3,$ поскольку для $n=2$ они справедливы для
обычных классов Соболева $W_{\rm loc}^{1, 1}$ и получены
в~\cite{Sev}.

\medskip
Напомним определения. Здесь и далее $G$ обозначает некоторую
фиксированную группу мёбиусовых отображений единичного шара на себя,
кроме того, точки $x$ и $y\in{\Bbb B}^n$ будут называться {\it
$G$-эквивалентными} (или короче, {\it эквивалентными}), если
найдётся $A\in G$ такое, что $x=A(y).$ Множество, состоящее из
классов эквивалентности элементов по указанному принципу,
обозначается через ${\Bbb B}^n/G.$ Согласно~\cite[п.~3.4]{MS$_1$},
{\it гиперболическая мера} измеримого по Лебегу множества $A\subset
{\Bbb B}^n$ определяется соотношением
\begin{equation}\label{eq2B}
v(A)=\int\limits_A\frac{2^n\, dm(x)}{{(1-|x|^2)}^n}\,.
\end{equation}
Определим {\it гиперболическое расстояние} $h(x, y)$ между точками
$x, y\in {\Bbb B}^n$ посредством соотношения
\begin{equation}\label{eq3}
h(x, y)=\log\,\frac{1+t}{1-t}\,,\quad
t=\frac{|x-y|}{\sqrt{|x-y|^2+(1-|x|^2)(1-|y|^2)}}\,,
\end{equation}
см., напр., \cite[соотношение~(2.18), замечание~2.12 и
упражнение~2.52]{Vu$_1$}. Отметим, что $h(x, y)=h(g(x), g(y))$ для
всякого $g\in {\cal{G M}}({\Bbb B}^n),$
см.~\cite[соотношение~(2.20)]{Vu$_1$}.
В дальнейшем обозначим через $I$ тождественное отображение в~${\Bbb
R}^n.$ Согласно~\cite[п.~3.4]{MS$_1$}, {\it нормальным
фундаментальным многогранником} с центром в точке $x_0$ называется
множество
\begin{equation}\label{eq1}
P=\{x\in {\Bbb B}^n: d(x, x_0)<d(x, T(x_0))\}
\end{equation}
при каждом $T\in G\setminus\{I\},$ где $d(x, y)$ обозначает
гиперболическое расстояние между точками $x, y\in {\Bbb B}^n$ (см.,
напр., \cite[соотношение~(2.18)]{Vu$_1$}). Пусть $\pi:{\Bbb
B}^n\rightarrow {\Bbb B}^n/G$ -- естественная проекция области
${\Bbb B}^n$ на фактор пространство ${\Bbb B}^n/G,$ тогда {\it
гиперболическая мера} $\widetilde{v}(A)$ множества $A\subset {\Bbb
B}^n/G$ определяется соотношением $v(P\cap\pi^{\,-1}(A)),$ $P$ --
нормальный фундаментальный многогранник вида~(\ref{eq1}). Прямыми
вычислениями нетрудно убедиться, что гиперболическая мера не
изменяется при всяком отображении~$g\in{\cal{G M}}({\Bbb B}^n).$

\medskip
Для элементов $p_1, p_2\in {\Bbb B}^n/G,$ положим
\begin{equation}\label{eq2}
\widetilde{h}(p_1, p_2):=\inf\limits_{g_1, g_2\in G}h(g_1(z_1),
g_2(z_2))\,,
\end{equation}
где $p_i=G_{z_i}=\{\xi\in {\Bbb B}^n:\,\exists\, g\in G:
\xi=g(z_i)\},$ $i=1,2.$ Заметим, что $\widetilde{h}$ определяет
метрику на ${\Bbb B}^n/G,$ если $G$ -- разрывная группа мёбиусовых
отображений единичного шара ${\Bbb B}^n,$ $n\geqslant 2,$ на себя,
не имеющая неподвижных точек в ${\Bbb B}^n$ (см.~\cite{Sev$_1$}).
Множество $G_{x_i}$ будем называть {\it орбитой} точки $x_i,$ а
$p_1$ и $p_2$ назовём {\it орбитами} точек $z_1$ и $z_2,$
соответственно. {\it Длина} кривой $\gamma:[a, b]\rightarrow {\Bbb
B}^n/G$ в фактор-пространстве ${\Bbb B}^n/G$ на участке $[a, t],$
$a\leqslant t\leqslant b,$ определяется следующим образом:
\begin{equation}\label{eq5D}
l_{\gamma}(t):=\sup\limits_{\pi}\sum\limits_{k=0}^m\widetilde{h}(\gamma(t_k),
\gamma(t_{k+1}))\,,
\end{equation}
где~$\sup$ берётся по всем разбиениям $\pi=\{a=t_0\leqslant
t_1\leqslant t_2\leqslant\ldots\leqslant t_m=b\}.$

\medskip
Будем говорить, что группа $G$ мёбиусовых преобразований $g$
единичного шара на себя действует разывно в ${\Bbb B}^n,$ если
каждая точка $x\in {\Bbb B}^n$ имеет окрестность $U$ такую, что
$g(U)\cap U=\varnothing$ для всех $g\in G,$ кроме, может быть,
конечного числа элементов. Будем говорить, что группа $G$ не имеет
{\it неподвижных точек} в ${\Bbb B}^n,$ если для произвольного $a\in
{\Bbb B}^n $ равенство $g(a)=a$ возможно лишь при $g=I.$

\medskip
Пусть $D,$ $D_{\,*}$ -- области на фактор-пространствах ${\Bbb
B}^n/G$ и ${\Bbb B}^n/G_{\,*},$ соответственно. Предположим, что
${\Bbb B}^n/G$ и ${\Bbb B}^n/G_{\,*}$ -- метрические пространства с
метриками  $\widetilde{h}$ и $\widetilde{h_*},$ соответственно (то,
что $\widetilde{h}$ и $\widetilde{h_*}$ являются метриками будет
установлено в следующем разделе). Элементы длины и объёма
обозначаются на ${\Bbb B}^n/G$ и ${\Bbb B}^n/G_*,$ соответственно,
$ds_{\widetilde{h}},$ $d\widetilde{v}$ и $ds_{\widetilde{h_*}},$
$d\widetilde{v_*}.$ Отображение $f:D\rightarrow D_{\,*}$ будет
называться {\it дискретным}, если прообраз $f^{-1}\left(y\right)$
каждой точки $y\in D_{\,*}$ состоит только из изолированных точек.
Отображение $f:D\rightarrow D_{\,*}$ будет называться {\it
открытым}, если образ любого открытого множества $U\subset D$
является открытым множеством в $D_{\,*}.$ Будем говорить, что~$f\in
W_{\rm loc}^{1,1}(D),$ если для каждой точки $p_0\in D$ найдутся
открытые окрестности $U$ и $W,$ содержащие точки $p_0$ и $f(p_0),$
соответственно, относительно которых естественные проекции
$\pi:(\pi^{-1}(U))\rightarrow U$ и $\pi_*:(\pi_*^{-1}(V))\rightarrow
V$ являются взаимно однозначными отображениями, при этом,
$F=\pi_*^{\,-1}\circ f\circ\pi\in W_{\rm loc}^{1,1}(\pi^{\,-1}(U)).$
Следующее утверждение установлено в~\cite[предложение~1.1]{Sev$_1$}.

\medskip
\begin{proposition}\label{pr1A}
{\sl Предположим, $G$ -- разрывная группа мёбиусовых отображений
единичного шара ${\Bbb B}^n,$ $n\geqslant 2,$ на себя, не имеющая
неподвижных точек в ${\Bbb B}^n.$ Тогда фактор-пространство ${\Bbb
B}^n/G$ является конформным многообразием, т.е, топологическим
многообразием, в котором любые две карты согласованы при помощи
конформных отображений. При этом, естественная проекция $\pi,$
отображающая ${\Bbb B}^n$ на ${\Bbb B}^n/G,$ является локальным
гомеоморфизмом. Более того, соответствующие пары вида $(U,
\pi^{\,-1}),$ где $U$ -- некоторая окрестность произвольной точки
$p\in {\Bbb B}^n/G,$ в которой отображение $\pi^{\,-1}$ определено и
непрерывно, могут быть рассмотрены как карты, соответствующие
указанному многообразию.}
\end{proposition}

\medskip В силу предложения~\ref{pr1A} имеет смысл говорить о
локальных координатах $(U, \varphi)$ точки $p_0\in U\subset D.$
Пишем $f\in W_{\rm loc}^{1,p}(D),$ $p\geqslant 1,$ если $f\in W_{\rm
loc}^{1,1}(D)$ и, кроме того, $\frac{\partial f_i}{\partial x_j}\in
L^p_{\rm loc}(D)$ в локальных координатах. Для отображения
$f:D\rightarrow{\Bbb R}^n,$ имеющего почти всюду в $D,$ частные
производные, полагаем
\begin{equation}\label{eq1E}\Vert
f^{\,\prime}(x)\Vert=\max\limits_{|h|=1}{|f^{\,\prime}(x)h|}\,,\quad
J(x, f)={\rm det}\, f^{\,\prime}(x)\,.\end{equation}
{\it Внешняя дилатация $K_O(x,f)$} отображения $f$ в точке $x$
определена соотношением
\begin{equation}\label{eq1B}
K_O(x,f)\quad =\quad\left\{
\begin{array}{rr}
\frac{\Vert
f^{\,\prime}(x)\Vert^n}{|J(x, f)|}, & J(x,f)\ne 0,\\
1,  &  f^{\,\prime}(x)=0, \\
\infty, & \text{в\,\,остальных\,\,случаях.}
\end{array}
\right.\,.
\end{equation}
Если речь идёт об отображении $f$ между областями $D$ и $D_*,$
принадлежащих фактор-пространствам ${\Bbb B}^n/G$ и ${\Bbb
B}^n/G_*,$ соответственно, то полагаем $K_O(p,f)=K_O(\varphi(p),
F),$ где $F=\psi\circ f\circ\varphi^{-\,1},$ $(U, \varphi)$ --
локальные координаты точки $p$ и $(V, \psi)$ -- локальные координаты
точки~$f(p).$ В качестве таких локальных координат могут
рассматриваться отображения $\pi$ и $\pi_*,$ соответственно. В силу
предложения~\ref{pr1A} данное определение корректно, т.е., не
зависит от выбора локальных координат, поскольку внутренняя
дилатация конформного отображения равна~1. Отображение
$f:D\rightarrow D_{\,*}$ будет называться {\it отображением с
конечным искажением,} если $f\in W_{\rm loc}^{1, 1}(D)$ и, кроме
того, найдётся почти всюду конечная функция $K(x),$ такая что в
локальных координатах $\Vert f^{\,\prime}(x)\Vert^n\leqslant
K(x)\cdot J(x, f)$ при почти всех $x\in G$ (где $G$ --
соответствующая область в ${\Bbb R}^n$).

\medskip
Пусть $D$ и $D_*$ -- области, принадлежащие фактор-пространствам
${\Bbb B}^n/G$ и ${\Bbb B}^n/G_*,$ соответственно, и пусть
$\varphi\colon [0,\infty)\rightarrow[0,\infty)$ -- неубывающая
функция. Будем говорить, что $f\in W^{1,\varphi}_{\rm loc}(D),$ если
для каждой точки $p\in D$ и каждой точки $f(p)$ найдутся окрестности
$U$ и $V$ этих точек, а также координатные отображения
$\pi:(\pi^{\,-1}(U))\rightarrow U$ и
$\pi_*:(\pi_*^{\,-1}(V))\rightarrow V,$ являющиеся взаимно
однозначными в $U$ и $V,$ соответственно, такие что
$F=\pi_*^{\,-1}\circ f\circ\pi\in W_{\rm loc}^{1,1}(\pi^{\,-1}(U)),$
при этом,
\begin{equation}\label{eq1AAA}
\int\limits_{\pi^{-1}(U)}\varphi\left(|\nabla
F(x)|\right)\,dm(x)<\infty\,,
\end{equation}
где, как обычно,
$$|\nabla
F(x)|=\sqrt{\sum\limits_{i=1}^n\sum\limits_{j=1}^n\left(\frac{\partial
F_i}{\partial x_j}\right)^2}\,.$$ Класс $W^{1,\varphi}_{\rm loc}$
называется классом {\it Орлича--Соболева}.

\medskip
Как обычно, кривая $\gamma$ на фактор-пространстве ${\Bbb B}^n/G$
определяется как непрерывное отображение $\gamma:I\rightarrow {\Bbb
B}^n/G,$ где $I$ -- конечный отрезок, интервал либо полуинтервал
числовой прямой. Пусть $\Gamma$ -- семейство кривых в ${\Bbb
B}^n/G.$ Борелевская функция $\rho:{\Bbb B}^n/G\rightarrow [0,
\infty]$ будет называться {\it допустимой} для $\Gamma,$ если
$\int\limits_{\gamma}\rho(p)\,ds_{\widetilde{h}}(p)\geqslant 1$ для
всякой (локально спрямляемой) кривой $\gamma\in \Gamma.$ Последнее
записывают в виде: $\rho\in {\rm adm}\,\Gamma.$ {\it Модулем}
семейства $\Gamma$ называется величина, определённая равенством
$$M(\Gamma):=\inf\limits_{\rho\in {\rm adm}\,\Gamma}\int\limits_{{\Bbb
B}^n/G}\rho^n(p)\,d\widetilde{v}(p)\,.$$
Всюду далее $\mathcal{H}^k$, $k=1,\ldots,n$ обозначает {\it
$k$-мер\-ную меру Хаусдорфа} на фактор-пространстве ${\Bbb B}^n/G$
относительно расстояния $\widetilde{h}.$ Точнее, если $A$~---
множество в ${\Bbb B}^n/G$, то
\begin{equation}\label{eq4}
\mathcal{H}_{\widetilde{h}}^k(A):=\sup_{\varepsilon>0}\
\mathcal{H}^k_{\widetilde{h},\,\varepsilon}(A)\,,
\end{equation}
где
\begin{equation}\label{eq5A}
\mathcal{H}^k_{\widetilde{h},\,\varepsilon}(A):= \Omega_k2^{\,-k}
\inf\sum^{\infty}_{i=1}\left(\widetilde{h}(A_i)\right)^k\,,
\end{equation}
$\Omega_k$ -- площадь единичного шара в ${\Bbb R}^k,$ а инфимум
в~(\ref{eq5A}) берётся по всем покрытиям $A$ множествами $A_i$ с
$\widetilde{h}(A_i)<\varepsilon.$ Отметим, что
$\mathcal{H}_{\widetilde{h}}^k$ является {\it внешней мерой в смысле
Каратеодори}, см.~\cite{Sa}. Если в этим определении заменить
$\widetilde{h}$ на гиперболическую метрику $h,$ то получим
хаусдорфову меру $\mathcal{H}^k_h,$ а если использовать евклидову
метрику $|\cdot|,$ то будем иметь обычную (евклидову) хаусдорфову
$k$-мерную меру~$\mathcal{H}^k.$  В этом случае, для выражения
из~(\ref{eq5A}) мы используем обозначения
$\mathcal{H}^k_{h,\,\varepsilon}(A)$ и
$\mathcal{H}^k_{\varepsilon}(A),$ соответственно. Отметим, что
$\mathcal{H}_{\widetilde{h}}^k$ является {\it внешней мерой в смысле
Каратеодори}, см.~\cite{Sa}.

\medskip
Пусть далее $\omega$ --- открытое множество в $\overline{{\Bbb
R}^k},$ $k=1,\ldots,n-1.$ Тогда {\it $k$-мер\-ной поверхностью} $S$
на фактор-пространстве ${\Bbb B}^n/G$ называется произвольное
непрерывное отображение $S\colon\omega\rightarrow{\Bbb B}^n/G.$ {\it
Функцией кратности} $N(S,y)$ поверхности $S$ называется число
прообразов $y \in {\Bbb B}^n/G.$ Другими словами, символ $N(S,y)$
обозначает кратность накрытия точки $y$ поверхностью $S.$ Хорошо
известно, что функция кратности является полунепрерывной снизу,
т.е., для каждой последовательности $y_m\in{\Bbb B}^n/G$,
$m=1,2,\ldots\,$, такой что $y_m\rightarrow y\in{\Bbb B}^n/G$ при
$m\rightarrow\infty,$ выполняется условие $N(S,y)\ \geqslant\
\liminf\limits_{m\rightarrow\infty}\:N(S,y_m),$ см., напр.,
\cite[с.~160]{RR}. Отсюда следует, что функция $N(S,y)$ является
измеримой по Борелю и, следовательно, измеримой относительно
произвольной хаусдорфовой меры $\mathcal{H}_{\widetilde{h}}^k,$ см.,
напр., \cite[теорему~II~(7.6)]{Sa}.

\medskip
Напомним, что {\it $k$-мер\-ной хаусдорфовой площадью}
борелевского множества $B$ в ${\Bbb B}^n/G$ (либо просто {\it
площадью $B$}\index{площадь $B$} при $k=n-1$), ассоциированной с
поверхностью $S\colon\omega\rightarrow {\Bbb B}^n/G,$ называем
величину
\begin{equation*} \label{Smolka_gl_12_eq8.2.4} {\mathcal {\widetilde{A}}}_S(B)\ =\
{\mathcal{\widetilde{A}}}^{k}_S(B)\ :=\ \int\limits_B N(S,y)\
d\mathcal{H}_{\widetilde{h}}^ky\,,
\end{equation*}
см., напр., \cite[разд.~3.2.1]{Fe}. Соответственно, для борелевской
функции $\rho\colon{\Bbb B}^n/G\rightarrow[0,\infty]$ её {\it
интеграл над поверхностью} $S$ определяем равенством
\begin{equation*}\label{Smolka_gl_12_eq8.2.5}
\int\limits_S \rho\ d{\mathcal{\widetilde{A}}}\ :=\
\int\limits_{{\Bbb B}^n/G}\rho(y)\:N(S,y)\
d\mathcal{H}_{\widetilde{h}}^ky\,.\end{equation*}

\medskip
Пусть $n\geqslant 2,$ и $\Gamma$ -- семейство $k$-мерных
поверхностей $S.$ Борелевскую функцию $\rho\colon{\Bbb
B}^n/G\rightarrow\overline{{\Bbb R}^+}$ будем называть {\it
допустимой} для семейства $\Gamma,$ сокр. $\rho\in{\rm
adm}\,\Gamma,$ если
\begin{equation}
\label{eq8.2.6}\int\limits_S\rho^k\,d{\mathcal{\widetilde{A}}}\geqslant
1\end{equation} для каждой поверхности $S\in\Gamma.$ Если речь идёт
о семействе поверхностей, то пока не оговорено противное, мы
полагаем $k=n-1.$
Модуль семейства $\Gamma$ поверхностей $S$ определяется соотношением
$$M(\Gamma)=\inf\limits_{\rho\in{\rm adm}\,\Gamma}
\int\limits_{{\Bbb B}^n/G}\rho^n(p)\,d\widetilde{v}(p)\,.$$
Заметим, что модуль $M(\cdot),$ определённый таким образом,
представляет собой внешнюю меру в пространстве $k$-мерных
поверхностей (см.~\cite{Fu}). Будем говорить, что некоторое свойство
$P$ выполнено для {\it почти всех поверхностей} области $D,$ если
оно имеет место для всех поверхностей, лежащих в $D,$ кроме, быть
может, некоторого их подсемейства, модуль которого равен нулю. Мы
будем говорить, что измеримая относительно гиперболической меры
$\widetilde{v}$ на ${\Bbb B}^n/G$ функция $\rho\colon{\Bbb
B}^n/G\rightarrow\overline{{\Bbb R}^+}$ {\it обобщённо допустима}
для семейства $\Gamma$ поверхностей $S$ в ${\Bbb B}^n/G,$ сокр.
$\rho\in{\rm ext}\,{\rm adm}\,\Gamma,$ если
соотношение~(\ref{eq8.2.6}) выполнено для почти всех поверхностей
$S$ семейства $\Gamma.$ {\it Обобщённый модуль} $\overline
M(\Gamma)$ семейства $\Gamma$ определяется равенством
$$\overline M(\Gamma)= \inf\int\limits_{{\Bbb B}^n/G}\rho^n(p)\,d\widetilde{v}(p)\,,$$
где точная нижняя грань берётся по всем функциям $\rho\in{\rm
ext}\,{\rm adm}\,\Gamma.$

\medskip
Изучение следующего класса отображений связано с кольцевым
определением квазиконформности по Герингу, а также с приложениями
модульных неравенств к классам Соболева и Орлича--Соболева (см.,
напр.~\cite{Ge$_3$} и \cite[глава~9]{MRSY}). Пусть $n\geqslant 2,$
$D$ и $D^{\,\prime}$~--- заданные области в ${\Bbb B}^n/G$ и ${\Bbb
B}^n/G_*,$ соответственно, $p_0\in\overline{D}$ и $Q\colon
D\rightarrow(0,\infty)$~--- измеримая функция относительно
гиперболической меры на~${\Bbb B}^n/G.$ Будем говорить, что $f\colon
D\rightarrow D^{\,\prime}$~--- {\it нижнее $Q$-отображение в точке}
$p_0,$ если существует окрестность $U$ точки $p_0$ такая, что
соотношение
\begin{equation}\label{eq1A}
M(f(\Sigma_{\varepsilon}))\geqslant \inf\limits_{\rho\in{\rm
ext\,adm}\,\Sigma_{\varepsilon}}\int\limits_{D\cap
\widetilde{A}(p_0, \varepsilon,
\varepsilon_0)}\frac{\rho^n(p)}{Q(p)}\,d\widetilde{v}(p)
\end{equation}
имеет место для каждого кольца
$$\widetilde{A}(p_0, \varepsilon, \varepsilon_0)=
\{p\in {\Bbb B}^n/G: \varepsilon<\widetilde{h}(p,
p_0)<\varepsilon_0\}$$ $\varepsilon_0\in(0,d_0),$
$d_0=\sup\limits_{p\in U}\widetilde{h}(p, p_0),$
где $\Sigma_{\varepsilon}$ обозначает семейство всех пересечений
геодезических сфер $S(p_0, r)$ с областью $D,$ $r\in (\varepsilon,
\varepsilon_0).$

\medskip
Пусть $D$ --- подмножество ${\Bbb B}^n/G.$ Для отображения $f\colon
D\,\rightarrow\,{\Bbb B}^n/G_*,$ множества $E\subset D$ и
$y\,\in\,{\Bbb B}^n/G_*$ определим {\it функцию кратности $N(y, f,
E)$} как число прообразов точки $y$ во множестве $E,$ т.е.,
\begin{equation}\label{eq1.7A}
N(y, f, E)\,=\,{\rm card}\,\left\{p\in E\,: f(p)=y\right\}\,, \qquad
N(f, E)\,=\,\sup\limits_{y\in {\Bbb B}^n/G_*}\,N(p, f, E)\,.
\end{equation}

\medskip
Имеет место следующее утверждение (см. также~\cite[теорема~5]{KRSS}
и~\cite[теорема~4.2]{IS}).

\medskip
 \begin{theorem}{}\label{thOS4.1}
{\sl Пусть $n\geqslant 3,$ $G$ и $G_*$ -- некоторые группы
мёбиусовых автоморфизмов единичного шара, действующие разрывно в
${\Bbb B}^n$ и не имеющие в ${\Bbb B}^n$ неподвижных точек.
Предположим, $D$ и $D_{\,*}$~--- области, принадлежащие ${\Bbb
B}^n/G$ и ${\Bbb B}^n/G_*$ соответственно, при этом, $\overline{D}$
и $\overline{D_{\,*}}$ являются компактами. Пусть также
$\varphi\colon(0,\infty)\rightarrow (0,\infty)$~--- неубывающая
функция, удовлетворяющая условию Кальдерона
\begin{equation}\label{eqOS3.0a}
\int\limits_{1}^{\infty}\left[\frac{t}{\varphi(t)}\right]^
{\frac{1}{n-2}}dt<\infty\,.
\end{equation}
Тогда каждое открытое дискретное отображение $f\colon D\rightarrow
{\Bbb B}^n/G_*$ класса $W^{1,\varphi}_{\rm loc},$ имеющее конечное
искажение и такое, что $N(f, D)<\infty,$ является нижним
$Q$-отображением в каждой точке $p_0\in\overline{D}$ при
$Q(p):=c\cdot N(f, D)\cdot K_O(p, f),$ где $c>0$ -- некоторая
постоянная, внешняя дилатация $K_O(p, f)$ отображения $f$ в точке
$p$ определена соотношением~(\ref{eq1B}), а кратность $N(f, D)$
определена в~(\ref{eq1.7A}).}
 \end{theorem}

\medskip
\begin{remark}
Условие~(\ref{eqOS3.0a}) принадлежит Кальдерону, см.~\cite{Cal}.
\end{remark}

\medskip
{\bf 2. Предварительные сведения.} Для точки $y_0\in {\Bbb B}^n$ и
числа $r\geqslant 0$ определим {\it гиперболический круг} $B_h(y_0,
r)$ и {\it гиперболическую окружность} $S_h(y_0, r)$ посредством
равенств
\begin{equation}\label{eq7}
B_h(y_0, r):=\{y\in {\Bbb B}^n: h(y_0, y)<r\}\,, S_h(y_0, r):=\{y\in
{\Bbb B}^n: h(y_0, y)=r\}\,. \end{equation}
В дальнейшем
\begin{equation}\label{eq2A} \widetilde{B}(p_0, r):=\{p\in {\Bbb B}^n/G: \widetilde{h}(p_0,
p)<r\}\,,\quad \widetilde{S}(p_0, r):=\{p\in {\Bbb B}^n/G:
\widetilde{h}(p_0, p)=r\}\end{equation}
-- шар и сфера с центром в точке $p_0$ на ${\Bbb B}^n/G.$ Всюду
далее $B(x_0, r)$ и $S(x_0, r)$ обозначают шар и сферу с центром в
точке $x_0\in {\Bbb R}^n.$

\medskip
Пусть $p_0\in {\Bbb B}^n/G$ и $z_0\in {\Bbb B}^n$ таково, что
$\pi(z_0)=p_0,$ где $\pi$ -- естественная проекция ${\Bbb B}^n$ на
${\Bbb B}^n/G.$ По предложению~1.1, и леммам~2.2 и 2.3
в~\cite{Sev$_1$} можно выбрать $\varepsilon_0>0$ таким, что
естественная проекция $\pi$ гомеоморфно
отображает~$B_h(z_0,\varepsilon_0)\subset {\Bbb B}^n$
на~$\widetilde{B}(p_0, \varepsilon_0)\subset {\Bbb B}^n/G$ так, что
$\overline{\widetilde{B}(p_0,\varepsilon_0)}$ -- компакт в ${\Bbb
B}^n/G.$ Не ограничивая общности, можно также считать, что $z_0=0$
(см.~\cite[пункт~1.34, лемма~1.37]{Vu$_1$}). В этом случае,
окрестность $U$ точки $p_0,$ лежащую вместе со своим замыканием в
$\widetilde{B}(p_0, \varepsilon_0),$ будем называть~{\it нормальной
окрестностью точки} $p_0.$ Имеет место следующий аналог теоремы
Фубини для фактор-пространств.

\medskip
\begin{lemma}\label{lem2}
{\sl Пусть $n\geqslant 2,$ $D$~--- область в ${\Bbb B}^n/G,$ где $G$
-- некоторая группа мёбиусовых автоморфизмов единичного шара,
действующих разрывно в ${\Bbb B}^n$ и не имеющих в ${\Bbb B}^n$
неподвижных точек.

Пусть $U$ -- некоторая нормальная окрестность точки $p_0\in {\Bbb
B}^n/G,$ $Q:U\rightarrow [0, \infty]$ -- измеримая относительно меры
$\widetilde{v}$ функция, $d_0:=\widetilde{h}(p_0,
\partial U):=\inf\limits_{p\in
\partial U}\widetilde{h}(p_0, p).$ Тогда существуют постоянные $C_1$ и $C_2,$ зависящие только
от $U$ такие, что при каждом $0<r_0\leqslant d_0$
\begin{equation}\label{eq6}
C_2\cdot \int\limits_0^{r_0}\int\limits_{\widetilde{S}(p_0,
r)}Q(p)\, d\mathcal{H}_{\widetilde{h}}^{n-1}\,dr\leqslant
\int\limits_{\widetilde{B}(p_0,
r_0)}Q(p)\,d\widetilde{v}(p)\leqslant C_1\cdot
\int\limits_0^{r_0}\int\limits_{\widetilde{S}(p_0, r)}Q(p)\,
d\mathcal{H}_{\widetilde{h}}^{n-1}\,dr\,,
\end{equation}
где $d\widetilde{v}(p)$ -- элемент объёма на ${\Bbb B}^n/G,$ а шар
$\widetilde{B}(p_0, r_0)$ и сфера $\widetilde{S}(p_0, r)$ определены
в~(\ref{eq2A}). }
\end{lemma}

\medskip
Утверждение леммы~\ref{lem2} включает в себя измеримость по $r$
внутренней функции $\psi(r):=\int\limits_{\widetilde{S}(p_0,
r)}Q(p)\, d\mathcal{H}_{\widetilde{h}}^{n-1}$ в правой части
интеграла в~(\ref{eq6}).

\begin{proof}
Согласно определению нормальной окрестности~$U$, шару
$\widetilde{B}(p_0, r_0)\subset {\Bbb B}^n/G$ соответствует шар
$B_h(0, r_0)\subset {\Bbb R}^n$ в гиперболической метрике $h.$
Учитывая соотношение~(\ref{eq7}), $B_h(0, r_0)=B\left(0,
\frac{e^{r_0}-1}{e^{r_0}+1}\right).$ По определению
\begin{equation}\label{eq8}
\int\limits_{\widetilde{B}(p_0,
r_0)}Q(p)\,\,d\widetilde{v}(p)=4\int\limits_{B\left(0,
\frac{e^{r_0}-1}{e^{r_0}+1}\right)}\frac{Q(\pi(x))}{{(1-|x|^2)}^n}\,\,dm(x):=I\,.
\end{equation}
Здесь и далее $\mathcal{H}^{n-1}$ обозначает $(n-1)$-мерную
хаусдорфову меру относительно обычной (евклидовой) метрики, а
$\mathcal{H}_h^{n-1}$ -- относительно гиперболической метрики $h$ в
единичном круге. Воспользуемся классической теоремой Фубини в
$n$-мерном евклидовом пространстве (см., напр.,
\cite[теорема~2.6.2]{Fe} либо \cite[теорема~8.1.III]{Sa}). Используя
полярные координаты и применяя эту теорему, будем иметь, что
\begin{equation}\label{eq9}I=4\int\limits_0^{\frac{e^{r_0}-1}{e^{r_0}+1}}\int\limits_{S(0,
r)} \frac{Q(\pi(x))}{{(1-|x|^2)}^n}\,d\mathcal{H}^{n-1}\,dr\,=
4\int\limits_0^{\frac{e^{r_0}-1}{e^{r_0}+1}}\frac{1}{(1-r^2)^n}\int\limits_{S(0,
r)}Q(\pi(x))\,d\mathcal{H}^{n-1}\,dr\,.
\end{equation}
Отметим, что утверждение теоремы Фубини включает в себя
существование внутренних интегралов в~(\ref{eq9}) относительно
элемента $d\mathcal{H}^{n-1},$ как минимум, при почти всех $r\in
\left[0, \frac{e^{r_0}-1}{e^{r_0}+1}\right],$ кроме того, согласно
этой теоремы функция $\alpha(r):=\int\limits_{S(0,
r)}Q(\pi(x))d\mathcal{H}^{n-1}$ измерима по $r$ (см. там же).
Учитывая, что $S_h\left(0, \log\frac{1+r}{1-r}\right)=S(0, r),$
последние соотношения можно переписать в виде
\begin{equation}\label{eq10A}I=4\int\limits_0^{\frac{e^{r_0}-1}{e^{r_0}+1}}\frac{1}{(1-r^2)^n}\int\limits_{S_h\left(0,
\log\frac{1+r}{1-r}\right)}Q(\pi(x))\,d\mathcal{H}^{n-1}\,dr\,.
\end{equation}
Делая в~(\ref{eq10A}) замену $\log\frac{1+r}{1-r}=t$ и учитывая
равенство $\frac{2}{1-r^2}dr=dt$ получаем, что
\begin{equation}\label{eq3A}
I=4\int\limits_0^{r_0}\frac{1}{2(1-r^2)^{n-1}}\int\limits_{S_h(0,
t)}Q(\pi(x))\,d\mathcal{H}^{n-1}\,dt\,,
\end{equation}
где $r=r(t)=\frac{e^t-1}{e^t+1}.$ Заметим, что
$\frac12\leqslant\frac{1}{2(1-r^2)^{n-1}}\leqslant C(r_0)$ при всех
$0<r\leqslant \frac{e^{r_0}-1}{e^{r_0}+1},$ поэтому из~(\ref{eq3A})
мы получаем, что
\begin{equation}\label{eq4A}
2\int\limits_0^{r_0}\int\limits_{S_h(0,
t)}Q(\pi(x))\,d\mathcal{H}^{n-1}\,dt\leqslant I\leqslant
4C(r_0)\cdot\int\limits_0^{r_0}\int\limits_{S_h(0,
t)}Q(\pi(x))\,d\mathcal{H}^{n-1}\,dt\,.
\end{equation}
Заметим, что при достаточно малом $r_0>0$ и некоторой постоянной
$c_1>0,$ зависящей только от $r_0,$
\begin{equation}\label{eq3B}
c_1\cdot h(z_1, z_2)\leqslant |z_1-z_2|\leqslant  h(z_1,
z_2)\quad\forall\,\, z_1, z_2\in B_h(0, r_0)\end{equation}
(см.~\cite[лемма~2.4]{Sev$_1$}). Заметим, что согласно~(\ref{eq3B})
измеримое множество $E\subset {\Bbb B}^n$ относительно
меры~$\mathcal{H}^{n-1}$ является также измеримым относительно
$\mathcal{H}_h^{n-1}.$ В самом деле, из определения
меры~$\mathcal{H}^{n-1}$ вытекает справедливость представления
$E=B\cup B_0,$ где $B$ является $\sigma$-компактом в $({\Bbb B}^n,
|\cdot|),$ $|\cdot|$ -- евклидова метрика, а
$\mathcal{H}^{n-1}(B_0)=0.$ Заметим также, что отображение $f(x)=x$
переводит ${\Bbb B}^n$ на себя так, что $f(B)$ снова является
$\sigma$-компактом в $({\Bbb B}^n, h),$ $h$ -- гиперболическая
метрика (это вытекает из непрерывности отображения $f(x)=x$ из
$({\Bbb B}^n, |\cdot|)$ в $({\Bbb B}^n, h),$ которая, в свою
очередь, вытекает из~(\ref{eq3B})). Наконец, $f(B_0)$ является
множеством меры ноль в~$({\Bbb B}^n, h)$ относительно $\mathcal{H
}_h^{n-1}$ ввиду соотношений~(\ref{eq3B}) и определения меры
$\mathcal{H}_h^{n-1}$ в~(\ref{eq4}). Значит, $f(E)=f(B)\cup f(B_0)$
является измеримым в~$({\Bbb B}^n, h)$ относительно $\mathcal{H
}_h^{n-1}$ как объединение двух множеств, одно из которых является
$\sigma$-компактом в $({\Bbb B}^n, h),$ а второе имеет меру ноль
относительно~$\mathcal{H}_h^{n-1}.$

\medskip
Из сказанного следует измеримость функции $Q(\pi(x))$ на $S_h(0, t)$
относительно меры $\mathcal{H}_h^{n-1}$ при почти всех $t.$ Кроме
того, по определению Хаусдорфовых мер в~(\ref{eq4}), для любого
измеримого множества $E\subset B_h(0, r_0)$ имеют место неравенства
\begin{equation}\label{eq1C}
c^{\,*}_1\cdot \mathcal{H}_h^{n-1}(E)\leqslant
\mathcal{H}^{n-1}(E)\leqslant
\mathcal{H}_h^{n-1}(E)\,,\end{equation}
где $c^{\,*}_1=c^{n-1}_1.$ Из~(\ref{eq1C}) по теореме
Радона-Никодима (см.~\cite[теорема~I.14.11]{Sa}, см.
также~\cite[теорема~2.3.VI]{KF}) вытекает, что
\begin{equation}\label{eq2C}
\mathcal{H}_h^{n-1}(S_h(0, t))=\int\limits_{S_h(0,
t)}\varphi(x)\,d\mathcal{H}^{n-1}\,,
\end{equation}
где $\varphi: S_h(0, t)\rightarrow {\Bbb R}$ -- некоторая
неотрицательная измеримая функция. Из~(\ref{eq2C}) по теореме о
преобразовании меры (см.~\cite[теорема~I.15.1]{Sa}) вытекает, что
\begin{equation}\label{eq2D}
\int\limits_{S_h(0,
t)}Q(\pi(x))\,d\mathcal{H}_h^{n-1}=\int\limits_{S_h(0,
t)}Q(x)\cdot\varphi(x)\,d\mathcal{H}^{n-1}\,.
\end{equation}
Отсюда по теореме Фубини (см., напр., \cite[теорема~2.6.2]{Fe} либо
\cite[теорема~8.1.III]{Sa}) функция $\psi(t)=\int\limits_{S_h(0,
t)}Q(\pi(x))d\mathcal{H}_h^{n-1}$ измерима по $t\in [0, r_0].$
Тогда из~(\ref{eq4A}) и (\ref{eq1C}) вытекает, что
\begin{equation}\label{eq5}
2c^{\,*}_1\cdot\int\limits_0^{r_0}\int\limits_{S_h(0,
t)}Q(\pi(x))\,d\mathcal{H}_h^{n-1}\,dt\leqslant I\leqslant
4C(r_0)\cdot\int\limits_0^{r_0}\int\limits_{S_h(0,
t)}Q(\pi(x))\,d\mathcal{H}_h^{n-1}\,dt\,.
\end{equation}
Поскольку отображение $\pi$ изометрично отображает $B_h(0, r_0)$ на
$\widetilde{B}(p_0, r_0),$ то для любого измеримого $E\subset B_h(0,
r_0)$ мы имеем:
\begin{equation}\label{eq11}
\mathcal{H}_h^{n-1}(E)=\mathcal{H}_{\widetilde{h}}^{n-1}(\pi(E))\,.
\end{equation}
Объединяя~(\ref{eq5}) и~(\ref{eq11}), мы получаем, что
\begin{equation}\label{eq6A}
2c^{\,*}_1\cdot\int\limits_0^{r_0}\int\limits_{\widetilde{S}(0,
t)}Q(p)\,d\mathcal{H}_{\widetilde{h}}^{n-1}\,dt\leqslant I\leqslant
4C(r_0)\cdot\int\limits_0^{r_0}\int\limits_{\widetilde{S}(0,
t)}Q(p)\,d\mathcal{H}_{\widetilde{h}}^{n-1}\,dt\,.
\end{equation}
Из неравенств~(\ref{eq6A}) вытекает~(\ref{eq6}), если
положить~$C_1:=2c^{\,*}_1$ и $C_2:=4C(r_0).$~$\Box$
\end{proof}

\medskip
Установим следующее утверждение, связывающее понятие <<почти всех>>
относительно модуля семейств поверхностей и лебеговском смысле (см.
также~\cite[лемма~4.1]{IS} для случая римановых многообразий).

\medskip
 \begin{lemma}\label{lem8.2.11}
{\sl\,Пусть $n\geqslant 2,$ $D$~--- область в ${\Bbb B}^n/G,$ где
$G$ -- некоторая группа мёбиусовых автоморфизмов единичного шара,
действующих разрывно в ${\Bbb B}^n$ и не имеющих в ${\Bbb B}^n$
неподвижных точек.

Предположим, $p_0\in\overline{D}$ и $U$ -- нормальная окрестность
точки $p_0.$ Если некоторое свойство $P$ имеет место для почти всех
сфер $D(p_0, r):=\widetilde{S}(p_0, r)\cap D,$ лежащих в $U,$ где
<<почти всех>> понимается в смысле модуля семейств поверхностей и,
кроме того, множество $$E=\{r\in {\Bbb R}: P\,\,\,\,\text{имеет\,\,
место для}\,\,\,\, \widetilde{S}(p_0, r)\cap D\}$$ измеримо по
Лебегу, то $P$ также имеет место для почти всех сфер $D(p_0, r),$
лежащих в $U$ относительно линейной меры Лебега по параметру $r\in
{\Bbb R }.$ Обратно, пусть $P$ имеет место для почти всех сфер
$D(p_0, r):=\widetilde{S}(p_0, r)\cap D$ относительно линейной меры
Лебега по $r\in {\Bbb R},$ тогда $P$ также имеет место для почти
всех сфер $D(p_0, r):=\widetilde{S}(p_0, r)\cap D$ в смысле модуля.}
 \end{lemma}

\medskip
 \begin{proof} {\it Необходимость.} Пусть некоторое свойство
$P$ имеет место для почти всех сфер $D(p_0, r):=\widetilde{S}(p_0,
r)\cap D,$ где <<почти всех>> понимается в смысле модуля семейств
кривых. Покажем, что $P$ также имеет место для почти всех сфер
$D(p_0, r)$ по отношению к параметру $r\in {\Bbb R}.$

Предположим, что заключение леммы не является верным. Тогда найдётся
семейство $\Gamma$ сфер $D(p_0, r)=\{p\in D: \widetilde{h}(p,
p_0)=r\},$ лежащих в некоторой нормальной окрестности $U$ точки
$p_0,$ для которого свойство $P$ выполнено в смысле почти всех
поверхностей относительно модуля, однако, нарушается для некоторого
множества индексов $r\in {\Bbb R}$ положительной меры. Обозначим
$\varphi:=(\pi|_U)^{\,-1}.$ Поскольку по определению окрестности $U$
найдётся $R>0$ такое, что $U\subset \widetilde{B}(p_0, R)$ и
$\overline{\widetilde{B}(p_0, R)}$ -- компакт в ${\Bbb B}^n/G,$ то
$\widetilde{h}(U)\leqslant
\widetilde{h}(\overline{\widetilde{B}(p_0, R)})<\infty.$

В дальнейшем мы условимся называть множество $B\subset {\Bbb B}^n/G$
борелевым (измеримым), если $\varphi(B)$ борелево (измеримо) в
${\Bbb B}^n.$ Ввиду регулярности меры Лебега $m_{1}$ найдётся
борелевское множество $B\subset {\Bbb R},$ такое что $m_{1}(B)>0$ и
свойство $P$ нарушается для почти всех $r\in B.$ Пусть $\rho\colon
{\Bbb B}^n/G\rightarrow[0,\infty]$~--- допустимая функция для
семейства $\Gamma.$ Заметим, что множество $E=\{p\in
\widetilde{B}(p_0, R)\,:\,\exists\,\, r\in B\cap [0, R]:\,
\widetilde{h}(p, p_0)=r\}$ является борелевым на ${\Bbb B}^n/G$,
поскольку $\varphi(E)=\{y\in B(0,
\frac{e^R-1}{e^R+1})\,:\,\exists\,\, r\in \varphi(B)\cap [0,
\frac{e^R-1}{e^R+1}] :\,|y| =r\},$ и $\varphi(E)$ измеримо по Борелю
ввиду \cite[теорема~2.6.2]{Fe}.

Учитывая, что $B$~--- борелево, мы можем считать, что $\rho\equiv 0$
вне множества $E,$ так как подобное предположение не нарушает
борелевость $\rho.$ По неравенству Гёльдера
$$\int\limits_{E}\rho^{n-1}(p)\ d\widetilde{v}(p)\ \leqslant\
\left(\int\limits_{E}\rho^n(p)\
d\widetilde{v}(p)\right)^{(n-1)/n}\left(\int\limits_{E}\
d\widetilde{v}(p)\right)^{1/n}$$
и следовательно, по лемме~\ref{lem2},
$$\int\limits_{{\Bbb B}^n/G}\rho^n(p)\ d\widetilde{v}(p)\ \geqslant\ \frac{\left(\int\limits_E\rho^{n-1}(p)\
d\widetilde{v}(p)\right)^{n/(n-1)}}{\left(\int\limits_E\
d\widetilde{v}(p)\right)^{1/(n-1)}}\ \geqslant\
\frac{(m_1(B))^{n/(n-1)}}{c}$$ для некоторого $c>0,$ т.е.,
$M(\Gamma)>0,$ что противоречит предположению леммы. Первая часть
леммы~\ref{lem8.2.11} доказана.

\medskip
{\it Достаточность.} Пусть $P$ имеет место для почти всех $r$
относительно меры Лебега и всех соответствующих этим $r$ сфер
$D(p_0, r),$ $r\in {\Bbb R}.$ Покажем, что $P$ также выполняется для
почти всех сфер $D(p_0, r):=S(p_0, r)\cap D$ в смысле модуля
семейств поверхностей.

Обозначим через $\Gamma_0$ семейство всех пересечений
$D_r:=D(p_0,r)$ сфер $S(p_0, r)$ с областью $D,$ для которых $P$ не
имеет места. Пусть $R$ обозначает множество всех $r\in {\Bbb R}$
таких, что $D_r\in\Gamma_0.$ Если $m_1(R)=0,$ то по лемме~\ref{lem2}
$\widetilde{v}(E)=0,$ где $E=\{p\in D\,|\, \widetilde{h}(p,
p_0)=r\in R\}.$ Рассмотрим функцию $\rho_1\colon{\Bbb
B}^n/G\rightarrow [0, \infty],$ равную $\infty$ при $p\in E,$ и
имеющую значение 0 во всех остальных точках. Отметим, что найдётся
борелева функция $\rho_2\colon{\Bbb B}^n/G\rightarrow [0, \infty],$
совпадающая почти всюду с $\rho_1$ (см.~\cite[раздел~2.3.5]{Fe}).
Таким образом, $M(\Gamma_0)\leqslant \int\limits_E \rho_2^n(p)
\,d\widetilde{v}(p)=\int\limits_E \rho_1^n(p)
\,d\widetilde{v}(p)=0,$ следовательно, $M(\Gamma_0)=0.$
Лемма~\ref{lem8.2.11} полностью доказана.~$\Box$
\end{proof}

\medskip
Следующее утверждение установлено в \cite[лемма~9.2]{MRSY}.

\medskip
 \begin{proposition}\label{Salnizh1}
{\sl Пусть $(X, \mu)$~--- измеримое пространство с конечной мерой
$\mu,$ $q\in(1,\infty),$ и пусть $\varphi\colon
X\rightarrow(0,\infty)$~--- измеримая функция. Полагаем
\begin{equation}\label{Sal_eq2.1.7}
I(\varphi, q)=\inf\limits_{\alpha}
\int\limits_{X}\varphi\,\alpha^q\,d\mu\,,
\end{equation}
где инфимум берется по всем измеримым функциям $\alpha\colon
X\rightarrow[0,\infty]$ таким\/{\em,} что
%
%\begin{equation*}%\label{Sal_eq2.1.8}
$\int\limits_{X}\alpha\,d\mu=1.$ %\end{equation*}
Тогда
%
%\begin{equation*}%\label{Sal_eq2.1.9}
$I(\varphi, q)=\left[\int\limits_{X}\varphi^{-\lambda}\,d\mu\right]
^{-\frac{1}{\lambda}},$ %\end{equation*}
где
%\begin{equation*}%\label{Sal_eq2.1.10}
$\lambda=\frac{q^{\,\prime}}{q},$ $\frac{1}{q}+\frac{1}{q^{\,\prime}}=1,$ %\end{equation*}
т.е. $\lambda=1/(q-1)\in(0,\infty).$ Точная нижняя грань
в~(\ref{Sal_eq2.1.7}) достигается на функции
$ \rho=\left(\int\limits_X
\varphi^\frac{1}{1-\alpha}\;d\mu\right)^{-1}\varphi^\frac{1}{1-\alpha}.
$}
\end{proposition}

\medskip
Следующее утверждение представляет собой критерий выполнения
бесконечной серии неравенств в~(\ref{eq1A}) (см. также
\cite[теорема~9.2]{MRSY} и \cite[лемма~4.2]{IS}).

\medskip
 \begin{lemma}\label{lem4A}
{\sl\, Пусть $D$ и $D_{\,*}$~--- заданные области в ${\Bbb B}^n/G$ и
${\Bbb B}^n/G_*,$ соответственно, $p_0\in\overline{D}$ и $Q\colon
D\rightarrow(0,\infty)$~--- заданная измеримая функция. Тогда, если
отображение $f\colon D\rightarrow D_{\,*}$ является нижним
$Q$-отображением в точке $p_0,$ то найдутся $0<d_0<\sup\limits_{p\in
D}\widetilde{h}(p, p_0)$ и постоянная $M>0,$ зависящая только от
окрестности $U$ такие, что
\begin{equation}\label{eq15}
M(f(\Sigma_{\varepsilon}))\geqslant M\cdot
\int\limits_{\varepsilon}^{\varepsilon_0}
\frac{dr}{\Vert\,Q\Vert_{n-1}(r)}\quad\forall\
\varepsilon\in(0,\varepsilon_0)\,,\ \varepsilon_0\in(0,d_0)\,,
\end{equation}
где $\Sigma_{\varepsilon}$ обозначает семейство всех пересечений
сфер $\widetilde{S}(p_0, r)$ с областью $D,$ $r\in (\varepsilon,
\varepsilon_0),$
$$
\Vert Q\Vert_{n-1}(r)=\left(\int\limits_{D(p_0,r)}
Q^{n-1}(p)\,\mathcal{H}^{n-1}_{\widetilde{h}}\right)^{\,1/(n-1)}$$
-- $L_{n-1}$-норма функции $Q$ в $D\cap
\widetilde{S}(p_0,r)=D(p_0,r)=\{p\in D\,:\, \widetilde{h}(p,
p_0)=r\}$.

Обратно, если соотношение~(\ref{eq15}) выполнено в некоторой
окрестности $U$ при некоторых $\varepsilon_0>0$ и $M>0,$ то $f$
является нижним $N\cdot Q$-отображением в точке $x_0,$ где $N>0$~---
также некоторая постоянная, зависящая только от окрестности~$U.$}
\end{lemma}

\medskip
 \begin{proof} Перед доказательством утверждения леммы выполним
некоторые необходимые преобразования (см. подпункты~\textbf{I} и
\textbf{II}).

\medskip
\textbf{I.} Не ограничивая общности, можно считать, что число $d_0$
из условия леммы таково, что $\overline{\widetilde{B}(p_0,
d_0)}\subset U,$ где $U$ -- некоторая нормальная окрестность точки
$p_0.$ Для удобства положим
$$P(p)=P(Q, \rho, p):=\frac{\rho^n(p)}{Q(p)},\quad \psi(r):=\inf\limits_{\alpha\in
I(r)}\int\limits_{D(p_0,r)}\frac{\alpha^n(p)}{Q(p)}\
\mathcal{H}^{n-1}_{\widetilde{h}}\,,$$
где $I(r)$ обозначает множество всех измеримых функций $\alpha$ на
сфере $D(p_0,r)=\widetilde{S}(p_0,r)\cap D$, таких что
$\int\limits_{D(p_0,r)}\alpha(p)
\mathcal{H}^{n-1}_{\widetilde{h}}=1.$
Мы можем также считать, что $\Vert Q\Vert(r)<\infty$ при почти всех
$r\in (\varepsilon, \varepsilon_0),$ поскольку в противном случае
утверждение леммы очевидно. Зафиксируем $\rho\in {\rm
ext\,adm}\,\Sigma_{\varepsilon},$ и положим
$A_{\rho}(r):=\int\limits_{D(p_0,r)}\rho^{n-1}(p)\
\mathcal{H}^{n-1}_{\widetilde{h}}$ и заметим, что $A_{\rho}(r)$ ---
измеримая по Лебегу функция относительно параметра $r$ ввиду
леммы~\ref{lem2}. Следовательно, множество $E\subset {\Bbb R}$ всех
таких $r>0,$ для которых
$A_{\rho}(r)=\int\limits_{D(p_0,r)}\rho^{n-1}(p)\
\mathcal{H}^{n-1}_{\widetilde{h}}\ne 0,$ измеримо по Лебегу и,
значит, по лемме~\ref{lem8.2.11}
$A_{\rho}(r)=\int\limits_{D(p_0,r)}\rho^{n-1}(p)\
\mathcal{H}^{n-1}_{\widetilde{h}}\ne 0$ также и при почти всех $r\in
(0, \varepsilon_0).$ Пусть $A_{\Sigma_{\varepsilon}}$ обозначает
класс всех измеримых по Лебегу функций $\rho\colon {\Bbb
B}^n/G\rightarrow\overline{{\Bbb R}^+},$ удовлетворяющих условию
$\int\limits_{D(p_0,
r)}\rho^{n-1}\,\mathcal{H}^{n-1}_{\widetilde{h}}=1$ для почти всех
$r\in (\varepsilon, \varepsilon_0).$ Для удобства обозначим
$D_{\varepsilon}:=D\cap \widetilde{A}(p_0, \varepsilon,
\varepsilon_0).$ Поскольку $A_{\Sigma_\varepsilon}\subset {\rm
ext\,adm}\,\Sigma_{\varepsilon},$  мы получим, что
\begin{equation}\label{eq1AA}
\inf\limits_{\rho\in{\rm
ext\,adm}\,\Sigma_{\varepsilon}}\int\limits_{D_{\varepsilon}}P(p)\,d\widetilde{v}(p)\leqslant
\inf\limits_{\rho\in
A_{\Sigma_\varepsilon}}\int\limits_{D_{\varepsilon}}P(p)\,d\widetilde{v}(p)\,.
\end{equation}

С другой стороны, для заданной функции $\rho\in {\rm
ext\,adm}\,\Sigma_{\varepsilon}$ положим $\beta_0(p):=\rho(p)\cdot
\left(\int\limits_{D(p_0,
r)}\rho^{n-1}(p)\mathcal{H}^{n-1}_{\widetilde{h}}\right)^{\frac{1}{1-n}}.$
Тогда $\rho \in A_{\Sigma_\varepsilon}.$ По лемме~\ref{lem8.2.11}
$A_{\rho}(r)\geqslant 1$ для почти всех $r\in (0, \varepsilon_0).$
Применяя лемму~\ref{lem2}, мы получим, что
$$\inf\limits_{\beta\in A_{\Sigma_\varepsilon}}\int\limits_{D_{\varepsilon}}\beta^n(p)Q^{\,-1}(p)\,
d\widetilde{v}(p)\leqslant $$
 \begin{multline}\label{eq2AB}
\leqslant\int\limits_{D_{\varepsilon}}\beta_0^n(p)Q^{\,-1}(p)\,d\widetilde{v}(p)\leqslant
C_1\cdot\int\limits_{\varepsilon}^{\varepsilon_0}\left(A_{\rho}(r)\right)^{n/(1-n)}
\int\limits_{D(p_0, r)}P(p)\,
\mathcal{H}^{n-1}_{\widetilde{h}}dr\leqslant\\
\leqslant
(C_1/C_2)\cdot\int\limits_{D_{\varepsilon}}P(p)\,d\widetilde{v}(p)\,,
\end{multline}
где $C_1$ и $C_2$ -- постоянные, соответствующие
неравенствам~(\ref{eq6}) для выбранной окрестности~$U.$ Здесь также
был использован тот факт, что точная нижняя грань любой величины не
превосходит любого её фиксированного значения.

\medskip
Из~(\ref{eq1AA}) и~(\ref{eq2AB}) следует, что
\begin{equation}\label{eq12}
(C_2/C_1)\cdot\inf\limits_{\rho\in
A_{\Sigma_\varepsilon}}\int\limits_{D_{\varepsilon}}P(p)\,d\widetilde{v}(p)
\leqslant\inf\limits_{\rho\in{\rm
ext\,adm}\,\Sigma_{\varepsilon}}\int\limits_{D_{\varepsilon}}P(p)\,d\widetilde{v}(p)\leqslant
\inf\limits_{\rho\in
A_{\Sigma_\varepsilon}}\int\limits_{D_{\varepsilon}}P(p)\,d\widetilde{v}(p)\,.
\end{equation}

\textbf{II.} Покажем теперь, что
\begin{equation}\label{eq9A}
C_2\cdot \int\limits_{\varepsilon}^{\varepsilon_0}\psi(r)\,dr
\leqslant \inf\limits_{\rho\in
A_{\Sigma_\varepsilon}}\int\limits_{D_{\varepsilon}}P(p)\
d\widetilde{v}(p)\leqslant C_1\cdot
\int\limits_{\varepsilon}^{\varepsilon_0}\psi(r)\,dr\,,
 \end{equation}
где $C_3, C_4>0$ -- некоторые постоянные.
Прежде всего, по предложению~\ref{Salnizh1}
\begin{equation}\label{eq14}
\psi(r)=(\Vert Q\Vert_{n-1}(r))^{\,-1}=\left(\int\limits_{D(p_0,r)}
Q^{n-1}(p)\,\mathcal{H}^{n-1}_{\widetilde{h}}\right)^{\,-1/(n-1)}\,.
\end{equation}
По лемме~\ref{lem2} функция $\psi(r)$ измерима по $r,$ так что
интегралы в~(\ref{eq9A}) определены корректно. Пусть $\rho\in
A_{\Sigma_\varepsilon}.$ Тогда по лемме~\ref{lem2} функция
$\rho_r(p):=\rho|_{\widetilde{S}(p_0, r)}$ измерима относительно
меры Хаусдорфа~$\mathcal{H}^{n-1}_{\widetilde{h}}$ при почти всех
$r\in (\varepsilon, \varepsilon_0).$  Используя определение точней
нижней грани, из~(\ref{eq6}) и~(\ref{eq14}) мы получаем, что
$$\int\limits_{D_{\varepsilon}}P(p)\,
d\widetilde{v}(p)\geqslant
C_2\cdot\int\limits_{\varepsilon}^{\varepsilon_0}\int\limits_{D(p_0,
r)}\rho_r^n(p)Q^{\,-1}(p) \,\mathcal{H}^{n-1}_{\widetilde{h}}dr
\geqslant
C_2\cdot\int\limits_{\varepsilon}^{\varepsilon_0}\psi(r)dr\,.$$
Переходя здесь к $\inf$ по всем $\rho\in A_{\Sigma_\varepsilon},$
будем иметь
\begin{equation}\label{eq11A}
\inf\limits_{\rho\in
A_{\Sigma_\varepsilon}}\int\limits_{D_{\varepsilon}}P(p)\
d\widetilde{v}(p)\geqslant
C_2\cdot\int\limits_{\varepsilon}^{\varepsilon_0} \psi(r)dr\,.
\end{equation}
Докажем теперь верхнее неравенство в~(\ref{eq9A}). Ввиду
предложения~\ref{Salnizh1} точная нижняя грань выражения
$\psi(\alpha,
r)=\int\limits_{D(p_0,r)}\alpha^{\frac{n}{n-1}}(p)Q^{\,-1}(p)\
\mathcal{H}^{n-1}_{\widetilde{h}}$ по всем $\alpha\in I(r)$
достигается на функции $\alpha_0(p):=
Q^{n-1}(p)\left(\int\limits_{D(p_0,r)}
Q^{n-1}(p)\,\mathcal{H}^{n-1}_{\widetilde{h}}\right)^{\,-1}.$
Заметим, что $\alpha_0\in A_{\Sigma_\varepsilon},$ поскольку по
предположению $\Vert Q\Vert(r)<\infty$ при почти всех $r\in
(\varepsilon, \varepsilon_0).$ Значит,
\begin{equation}\label{eq13}
\inf\limits_{\rho\in
A_{\Sigma_\varepsilon}}\int\limits_{D_{\varepsilon}}P(p)\
d\widetilde{v}(p)\leqslant
\int\limits_{D_{\varepsilon}}\alpha_0^{\frac{n}{n-1}}(p)Q^{\,-1}(p)\
d\widetilde{v}(p)\leqslant C_1\cdot
\int\limits_{\varepsilon}^{\varepsilon_0} \psi(r)\,dr\,.
\end{equation}
Из~(\ref{eq11A}) и~(\ref{eq13}) следует~(\ref{eq9A}).~$\Box$

\medskip
\textbf{III.} Докажем теперь утверждение леммы.

\medskip
а) Пусть $f$~--- нижнее $Q$-отображение в точке $p_0.$ Тогда, по
определению, $f$ удовлетворяет соотношению~(\ref{eq1A}). Однако,
ввиду соотношений~(\ref{eq12}) и~(\ref{eq9A}) имеем также
неравенство~(\ref{eq15}) с некоторой постоянной
$M:=\frac{C_2^2}{C_1}.$

\medskip
б) Пусть, напротив, мы имеем соотношение~(\ref{eq15}) c некоторой
постоянной $M>0$ в фиксированной окрестности $U$ точки $p_0.$ Тогда
ввиду соотношений~(\ref{eq12}) и~(\ref{eq9A}) отображение $f$
является нижним $M\cdot Q$-отображением в точке $p_0,$ где
$N:=C_1/M.$ Лемма доказана.~$\Box$
\end{proof}

\medskip
{\bf 3.~Доказательство теоремы~\ref{thOS4.1}}. Поскольку $f$
открыто, то отображение $f$ дифференцируемо почти всюду в $D$
локальных координатах (см.~\cite[теорема~1]{KRSS}). Пусть $B$ --
борелево множество всех точек $p\in D,$ где $f$ имеет полный
дифференциал $f^{\,\prime}(p)$ и $J(p, f)\ne 0$ в локальных
координатах. Заметим, что $B$ может быть представлено в виде не
более, чем счётного объединения борелевских множеств $B_l$,
$l=1,2,\ldots\,,$ таких что $f_l=f|_{B_l}$ являются билипшецевыми
гомеоморфизмами (см. \cite[пункты~3.2.2, 3.1.4 и 3.1.8]{Fe}), см.
рисунок~\ref{fig4} для иллюстрации.
\begin{figure}[h]
\centerline{\includegraphics[scale=0.5]{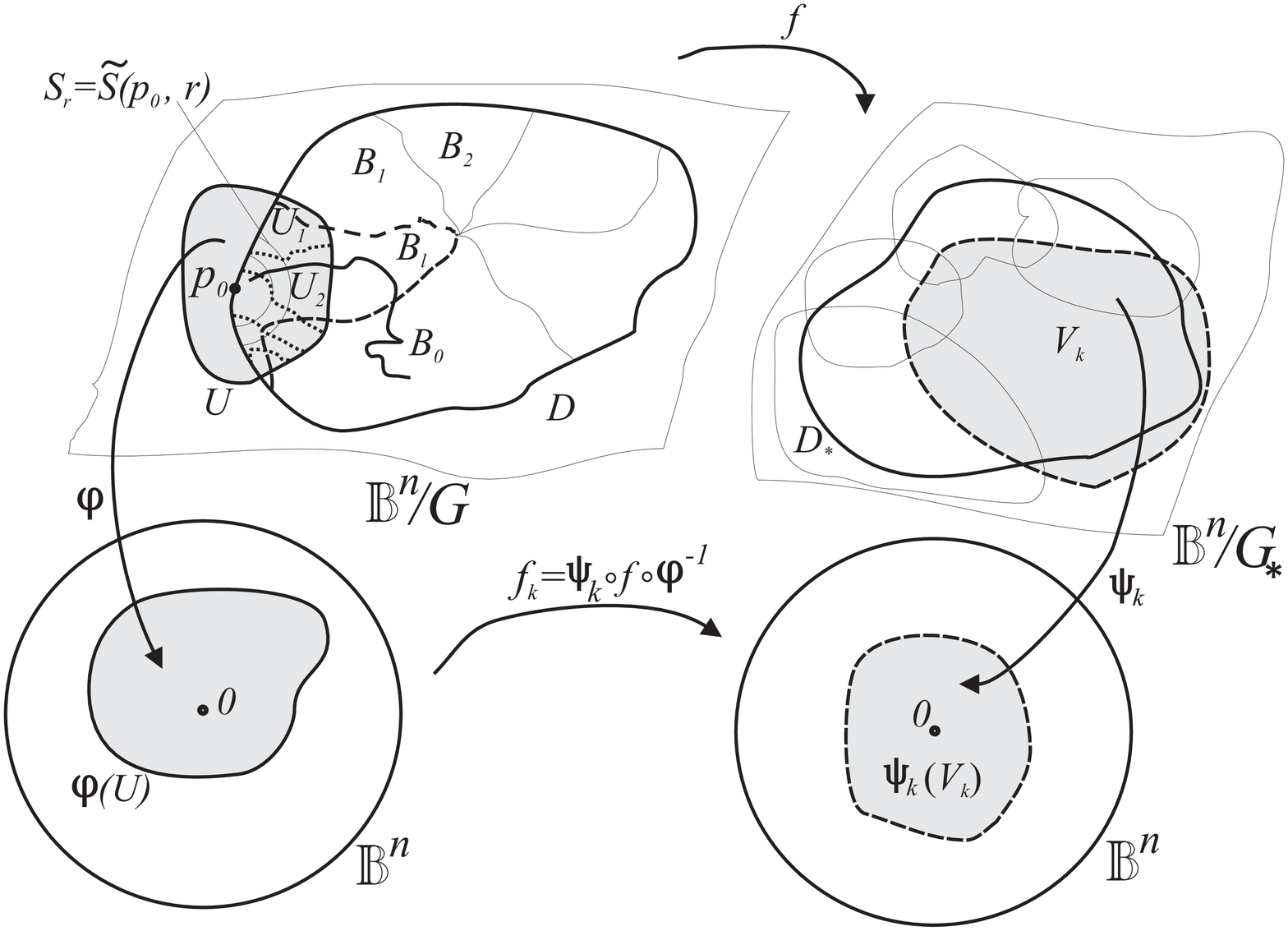}} \caption{К
доказательству теоремы~\ref{thOS4.1}}\label{fig4}
\end{figure}
Без ограничения общности, мы можем считать, что множества $B_l$
попарно не пересекаются. Обозначим также символом $B_*$ множество
всех точек $p\in D,$ где $f$ имеет полный дифференциал и
$f^{\,\prime}(p)=0.$

Поскольку $f$ -- конечного искажения, $f^{\,\prime}(p)=0$ для почти
всех точек $p,$ где $J(p, f)=0.$ Таким образом, по построению
множество $B_0:=D\setminus \left(B\bigcup B_*\right)$ имеет нулевую
$\widetilde{v}$-меру. Пусть $U$ -- нормальная окрестность точки
$p_0$ и $\varphi:U\rightarrow {\Bbb B}^n$ -- отображение,
соответствующее этой нормальной окрестности. Мы можем считать, что
$\varphi(U)\subset B(0, r_0),$ $0<r_0<1.$ Поскольку $\overline{D_*}$
-- компакт в ${\Bbb B}^n/G_*,$ мы можем покрыть $\overline{D_*}$
конечным числом нормальных окрестностей $V_k,$ $k=1,2,\ldots, m,$
таких, что $\psi_k: V_k\rightarrow B(0, R_k),$ $0<R_k<1,$ и $\psi_k$
-- отображения, соответствующие определению нормальной окрестности.
Пусть $R_0:=\max\limits_{1\leqslant k\leqslant m}R_k.$ Тогда ввиду
непрерывности отображения $f$ множества
$U^{\,\prime}_k:=f^{\,-1}(V_k\cap D_{\,*})\cap U$ являются открытыми
в $U$ и отображение
$$f_k:=\psi_k\circ f\circ\varphi^{\,-1}$$
является отображением из $\varphi(U^{\,\prime}_k)\subset B(0, r_0)$
в $\psi_k(V_k)\subset B(0, R_0).$

\medskip
Положим теперь $U_1=U^{\,\prime}_1,$ $U_2=U^{\,\prime}_2\setminus
U^{\,\prime}_1,$ $U_3=U^{\,\prime}_3\setminus (U^{\,\prime}_1\cup
U^{\,\prime}_2),$ $\ldots, U_m=U^{\,\prime}_m\setminus
(U^{\,\prime}_1\cup U^{\,\prime}_2\ldots U^{\,\prime}_{m-1}).$
Заметим, что по определению $U_m\subset U^{\,\prime}_m$ при
$m\geqslant 1$ и $U_s\cap U_k=\varnothing$ при $s\ne k.$ Пусть
$\Gamma$ -- семейство $D_r:=D\cap S_r$ всех пересечений сфер
$S_r=\widetilde{S}(p_0, r),$ $r\in(\varepsilon, r_0),$ с областью
$D.$ Зафиксируем допустимую функцию $\rho_*\in{\rm adm}\,f(\Gamma),$
$\rho_*\equiv 0$ вне $f(D)$, и положим $\rho\equiv 0$ вне $U$ и на
$B_0,$ и
$$\rho(p)\colon=\rho_*(f(p))\Vert f_k^{\,\prime}
(\varphi(p))\Vert \qquad\text{при}\ p\in U_k\setminus B_0\,,$$
где матричная норма производной $\Vert g^{\,\prime}(x)\Vert$
определена первым соотношением в~(\ref{eq1E}). Заметим, что
$$D_r:=D\cap
S_r=\left(\bigcup\limits_{1\leqslant k\leqslant m\atop 1\leqslant
l<\infty} D^{\,r}_{kl}\right)\bigcup \left(\bigcup\limits_{k=1}^m
S_r\cap U_k\cap B_*\right)\bigcup \left(\bigcup\limits_{k=1}^m
S_r\cap U_k\cap B_0\right)\,,$$
где $D^{\,r}_{kl}=S_r\cap U_k\cap B_l.$ %Тогда также
%
%$$D_{r}^{\,*}=f(S_r\cap D)=$$$$=\left(\bigcup\limits_{1\leqslant k\leqslant m\atop 1\leqslant
%l<\infty} f(D^{\,r}_{kl})\right)\bigcup \left(\bigcup\limits_{k=1}^m
%f(S_r\cap U_k\cap B_*)\right)\bigcup \left(\bigcup\limits_{k=1}^m
%f(S_r\cap U_k\cap B_0)\right)\,.$$
Положим $D^{\,*}_r=f(D_r).$ По~\cite[теорема~4, следствие~4]{KRSS},
учитывая неравенства~(\ref{eq1C}), имеем:
$\mathcal{H}_{\widetilde{h}}^{n-1}(f(U_k\cap B_0\cap D_r))=0$ и
$\mathcal{H}_{\widetilde{h}}^{n-1}(f(U_k\cap B_*\cap D_r))=0$ при
каждом $1\leqslant k\leqslant m$ и почти всех $r\in (0, r_0).$
Тогда, используя неравенства в~(\ref{eq1C}), при почти всех $r\in
(0, r_0)$ будем иметь:
$$1\leqslant\int\limits_{D^{\,*}_r}\rho^{n-1}_*(p_*)\cdot N(p_*, f,
D_r)\,d\mathcal{H}_{\widetilde{h_*}}^{n-1}(p_*)\leqslant$$$$\leqslant
\sum\limits_{k=1}^m
\sum\limits_{l=1}^{\infty}\,\,\,\int\limits_{\psi_k(f(D^{\,r}_{kl}))}
\rho^{n-1}_*(\psi_k^{\,-1}(y))\cdot N(y, \psi_k\circ f, D^{\,r}_{kl}
)\,d\mathcal{H}_h^{n-1}(y)\leqslant$$
\begin{equation}\label{eq27}\leqslant\frac{1}{c^*_1}\sum\limits_{k=1}^m
\sum\limits_{l=1}^{\infty}\,\,\,\int\limits_{\psi_k(f(D^{\,r}_{kl}))}
\rho^{n-1}_*(\psi_k^{\,-1}(y))N(y, \psi_k\circ f,
D^{\,r}_{kl})\,d\mathcal{H}^{n-1}(y)\,,
\end{equation}
где $c^*_1$ -- постоянная, участвующая в~(\ref{eq1C}) и
соответствующая $E=B_h(0, r_0).$ Рассуждая на каждом
множестве~$D^{\,r}_{kl}$ в отдельности и используя
\cite[пункт~1.7.6, теорема~2.10.43 и теорема~3.2.5]{Fe}, мы
получаем, что
$$\int\limits_{D^{\,r}_{kl}}\rho^{n-1}(p)\,d\mathcal{H}_{\widetilde{h}}^{n-1}(p)
=\int\limits_{\varphi(D^{\,r}_{kl})}\rho^{n-1}_*(\psi^{\,-1}_k(f_k(x))\Vert
f_k^{\,\prime}(x)\Vert^{n-1}\,\,d\mathcal{H}_{\widetilde{h}}^{n-1}(x)\geqslant$$
\begin{equation}\label{eq29}
\geqslant\int\limits_{\varphi(D^{\,r}_{kl})}\rho^{n-1}_*(\psi^{\,-1}_k(f_k(x))\Vert
f_k^{\,\prime}(x)\Vert^{n-1}\,\,d\mathcal{H}^{n-1}(x)\geqslant$$$$\geqslant
\int\limits_{\psi_k(f(D^{\,r}_{kl}))}\rho^{n-1}_*(\psi_k^{\,-1}(y))\cdot
N(y, \psi_k\circ f, D^{\,r}_{kl}) \,d\mathcal{H}^{n-1}(y)\,.
\end{equation}
Суммируя~(\ref{eq29}) по всем $1\leqslant k\leqslant m$ и
$1\leqslant l<\infty,$ и учитывая~(\ref{eq27}), а также
лемму~\ref{lem8.2.11}, заключаем, что
$\rho/c^*_1\in{\rm{ext\,adm}}\,\Gamma.$

\medskip
Используя замену переменных на каждом $B_l$, $l=1,2,\ldots$ (см.,
напр., \cite[теорема~3.2.5]{Fe}), свойство счётной аддитивности
интеграла Лебега, а также учитывая~(\ref{eq1}), получаем оценку:
$$\frac{1}{c_1^{*n}}\int\limits_{D}\frac{\rho^n(p)}{K_O(p, f)}\,d\widetilde{v}(p)=$$$$=
\frac{2^n}{c_1^{*n}}\sum\limits_{k=1}^m
\sum\limits_{l=1}^{\infty}\int\limits_{\varphi(U_k\cap
B_l)}\frac{\rho^n_*((f\circ \varphi^{-1})(x))\Vert f_k^{\,\prime}
(x)\Vert^n}{(1-|x|^2)^nK_O(\varphi^{\,-1}(x), f)}\,dm(x)\leqslant$$
$$\leqslant\frac{2^n}{c_1^{*n}(1-r^2_0)^n}\sum\limits_{k=1}^m
\sum\limits_{l=0}^{\infty}\int\limits_{{\Bbb
B}^n/G_*}\rho^n_*(\psi_k^{\,-1}(y))N(y, f_k, \varphi(U_k\cap
B_l))\,dm(y)\leqslant $$$$ \leqslant\frac{2^n}{c_1^{*n}(1-r^2_0)^n}
\sum\limits_{k=1}^m\int\limits_{{\Bbb
B}^n/G_*}\rho^n_*(\psi_k^{\,-1}(y))N(y, f_k,
\varphi(U_k))\,dm(y)\leqslant$$$$\leqslant
\frac{2^n}{c_1^{*n}(1-r^2_0)^n}
\sum\limits_{k=1}^m\int\limits_{{\Bbb
B}^n/G_*}\frac{\rho^n_*(\psi_k^{\,-1}(y))N(\psi^{\,-1}_k(y), f,
U_k)}{(1-|y|^2)^n}\,dm(y)=$$
$$=\frac{1}{c_1^{*n}(1-r^2_0)^n}
\sum\limits_{k=1}^m\int\limits_{{\Bbb B}^n/G_*}\rho^n_*(p_*)N(p_*,
f, U_k)\,d\widetilde{v_*}(p_*)\leqslant$$
$$\leqslant\frac{N(f,
D)}{c_1^{*n}(1-r^2_0)^n} \int\limits_{{\Bbb
B}^n/G_*}\rho^n_*(p_*)\,d\widetilde{v_*}(p_*)\,.$$
Для завершения доказательства следует положить
$c:=\frac{1}{c_1^{*n}(1-r^2_0)^n}.~\Box$

\medskip
{\bf 4. О гиперболическом и евклидовом модулях семейств кривых и
поверхностей}. Для дальнейшего исследования нам полезно будет
использовать следующие небольшие договорённости.

\medskip
Борелевская функция $\rho:{\Bbb B}^n\rightarrow [0, \infty]$ будет
называться {\it допустимой для $\Gamma$  гиперболическом смысле},
пишем $\rho\in {\rm adm}_h\,\Gamma,$ если
$\int\limits_{\gamma}\rho(x)\,ds_h(x)\geqslant 1$ для всякой
(локально спрямляемой) кривой $\gamma\in \Gamma,$ где $ds_h$ --
элемент длины, соответствующий гиперболической метрике $h$
в~(\ref{eq3}). {\it Модулем семейства $\Gamma$ в гиперболическом
смысле} называется величина, определённая равенством
$$M_h(\Gamma):=\inf\limits_{\rho\in {\rm adm}_h\,\Gamma}\int\limits_{{\Bbb
B}^n}\rho^n(x)\,dv(x)\,.$$
Для борелевской функции $\rho\colon{\Bbb B}^n\rightarrow[0,\infty]$
её {\it интеграл над $(n-1)$-мерной поверхностью $S$ в
гиперболическом смысле} определяем равенством
\begin{equation*}
\int\limits_S \rho\ d\mathcal{A}_h\ :=\ \int\limits_{{\Bbb
B}^n}\rho(y)\:N(S,y)\ d\mathcal{H}_h^{n-1}y\,,\end{equation*}
где $\mathcal{H}_h^{n-1}$ есть хаусдорфова $(n-1)$-мерная мера,
определённая при помощи гиперболической метрики $h.$ Борелевскую
функцию $\rho\colon{\Bbb B}^n\rightarrow\overline{{\Bbb R}^+}$ будем
называть {\it допустимой для семейства $\Gamma$ в гиперболическом
смысле},  сокр.~$\rho\in{\rm adm}_h\,\Gamma,$ если
\begin{equation*}
\int\limits_S\rho^{n-1}\,d\mathcal{A}_h\geqslant 1\end{equation*}
для каждой поверхности $S\in\Gamma.$
Модуль семейства $\Gamma$ поверхностей $S$ в гиперболическом смысле
определяется соотношением
$$M_h(\Gamma)=\inf\limits_{\rho\in{\rm adm}_h\,\Gamma}
\int\limits_{{\Bbb B}^n}\rho^n(x)\,dv(x)\,.$$

\medskip
Аналогично определяется модуль $M_e$ семейств кривых и поверхностей
в евклидовом смысле, а также семейство допустимых функций ${\rm
adm}_e\,\Gamma$ для семейств кривых (поверхностей) $\Gamma$ в~${\Bbb
B}^n.$ А именно, борелевская функция $\rho:{\Bbb B}^n\rightarrow [0,
\infty]$ называется {\it допустимой для $\Gamma$ в евклидовом
смысле}, пишем $\rho\in {\rm adm}_e\,\Gamma,$ если
$\int\limits_{\gamma}\rho(x)\,|dx|\geqslant 1$ для всякой (локально
спрямляемой) кривой $\gamma\in \Gamma,$ где $|dx|$ -- элемент
евклидовой длины. {\it Модулем семейства $\Gamma$ в евклидовом
смысле} называется величина, определённая равенством
$$M_e(\Gamma):=\inf\limits_{\rho\in {\rm adm}_e\,\Gamma}\int\limits_{{\Bbb
B}^n}\rho^n(x)\,dm(x)\,.$$
Для борелевской функции $\rho\colon{\Bbb B}^n\rightarrow[0,\infty]$
её {\it интеграл над $(n-1)$-мерной поверхностью $S$ в евклидовом
смысле} определяем равенством
\begin{equation*}
\int\limits_S \rho\ d\mathcal{A}\ :=\ \int\limits_{{\Bbb
B}^n}\rho(y)\:N(S,y)\ d\mathcal{H}^{n-1}y\,,\end{equation*}
где $\mathcal{H}^{n-1}$ есть хаусдорфова $(n-1)$-мерная мера,
определённая при помощи евклидовой метрики $|\cdot|.$ Борелевскую
функцию $\rho\colon{\Bbb B}^n\rightarrow\overline{{\Bbb R}^+}$ будем
называть {\it допустимой для семейства $\Gamma$ в евклидовом
смысле}, сокр.~$\rho\in{\rm adm}_e\,\Gamma,$ если
\begin{equation*}
\int\limits_S\rho^{n-1}\,d\mathcal{A}\geqslant 1\end{equation*} для
каждой поверхности $S\in\Gamma.$
Модуль семейства $\Gamma$ поверхностей $S$ в евклидовом смысле
определяется соотношением
$$M_e(\Gamma)=\inf\limits_{\rho\in{\rm adm}_e\,\Gamma}
\int\limits_{{\Bbb B}^n}\rho^n(x)\,dm(x)\,.$$

\medskip
{\bf 5. О верхних оценках искажения модуля.} Для простоты рассмотрим
случай, когда группа $G_*,$ соответствующая второму пространству
${\Bbb B}^n/G_*,$ состоит из одного отображения $g(x)=x$ (другими
словами, ${\Bbb B}^n/G_*$ -- единичный круг ${\Bbb B}^n$ с
гиперболической метрикой $h$ и гиперболическим объёмом $v$
(см.~(\ref{eq2B}) и~(\ref{eq3})).

\medskip
Пусть $n\geqslant 2,$ $p_0\in{\Bbb B}^n/G,$ где $G$ -- некоторая
группа мёбиусовых автоморфизмов единичного шара, действующих
разрывно в ${\Bbb B}^n$ и не имеющих в ${\Bbb B}^n$ неподвижных
точек, $Q:U\rightarrow [0, \infty]$ -- измеримая относительно меры
$\widetilde{v}$ функция. Обозначим
\begin{equation}\label{eq32*}
q_{p_0}(r):=\frac{1}{\omega_{n-1}r^{n-1}}\int\limits_{\widetilde{S}(p_0,
r)}Q(p)\,d\mathcal{H}_{\widetilde{h}}^{n-1}\,,
\end{equation}
где $\omega_{n-1}$ -- площадь единичной сферы~${\Bbb S}^{n-1}$ в
${\Bbb R}^n.$ По лемме~\ref{lem2} интеграл в~(\ref{eq32*}) определён
корректно при почти всех $0<r_0<\widetilde{h}(p_0,
\partial U),$ кроме того, по той же лемме $q_{p_0}(r)$ измерима по
$r.$ Начнём со следующего простого утверждения (см.
также~\cite[лемма~7.4]{MRSY}).

\medskip
\begin{proposition}\label{pr1B}
{\sl\, Пусть $n\geqslant 2,$ $p_0\in{\Bbb B}^n/G,$ где $G$ --
некоторая группа мёбиусовых автоморфизмов единичного шара,
действующих разрывно в ${\Bbb B}^n$ и не имеющих в ${\Bbb B}^n$
неподвижных точек.

Пусть $U$ -- некоторая нормальная окрестность точки $p_0\in {\Bbb
B}^n/G,$ $Q:U\rightarrow [0, \infty]$ -- измеримая относительно меры
$\widetilde{v}$ функция, интегрируемая в $U$ относительно меры
$\widetilde{v}.$ Положим
\begin{equation}\label{eq8B}
 \eta_0(r)=\frac{1}{Irq_{p_0}^{\frac{1}{n-1}}(r)}\,,
\end{equation}
где
\begin{equation}\label{eq9B}
I=I(p_0,r_1,r_2)=\int\limits_{r_1}^{r_2}\
\frac{dr}{rq_{p_0}^{\frac{1}{n-1}}(r)}\,.
\end{equation}
Тогда существуют постоянные $M_1$ и $M_2>0,$ зависящие только от
окрестности $U,$ такие, что
\begin{equation}\label{eq10B}
\frac{\omega_{n-1}}{I^{n-1}}\leqslant
M_1\cdot\int\limits_{\widetilde{A}} Q(p)\cdot
\eta_0^n(\widetilde{h}(p, p_0))\ d\widetilde{v}(p)\leqslant
M_2\cdot\int\limits_{\widetilde{A}} Q(p)\cdot
\eta^n((\widetilde{h}(p, p_0)))\ d\widetilde{v}(p)\,,
\end{equation}
где $\widetilde{A}=\widetilde{A}(p_0, r_1, r_2)=\{p\in {\Bbb B}^n/G:
r_1<\widetilde{h}(p, p_0)<r_2\},$ а $\eta: (r_1,r_2)\rightarrow
[0,\infty]$ -- произвольная измеримая по Лебегу функция,
удовлетворяющая условию
\begin{equation}\label{eq8AA}
\int\limits_{r_1}^{r_2}\eta(r)\,dr=1\,.
\end{equation}
}
\end{proposition}
\begin{proof}
Прежде всего, установим первое слева неравенство в~(\ref{eq10B}).
Применяя лемму~\ref{lem2}, получаем:
$$\int\limits_{\widetilde{A}} Q(p)\cdot
\eta_0^n(\widetilde{h}(p, p_0))\ d\widetilde{v}(p)\geqslant C_2\cdot
\int\limits_{r_1}^{r_2}\int\limits_{\widetilde{S}(p_0,
r)}Q(p)\cdot\eta_0^n(\widetilde{h}(p, p_0))\
\,d\mathcal{H}_{\widetilde{h}}^{n-1}dr=$$
\begin{equation}\label{eq16}
=C_2\cdot
\int\limits_{r_1}^{r_2}\left(\frac{1}{I^nr^nq_0^{n/(n-1)}(r)}\cdot\omega_{n-1}r^{n-1}\cdot\left(\frac{1}{\omega_{n-1}r^{n-1}}\cdot
\int\limits_{\widetilde{S}(p_0, r)}Q(p)
\,d\mathcal{H}_{\widetilde{h}}^{n-1}\right)\right)dr=
\end{equation}
$$=\frac{C_2\cdot\omega_{n-1}}{I^{n-1}}\,.$$
Здесь $C_2$ -- постоянная из леммы~\ref{lem2}, соответствующая
шару~$\widetilde{B}(p_0, d_0),$ $d_0:=\widetilde{h}(p_0, \partial
U),$ и зависящая только от окрестности $U.$ Первое неравенство
в~(\ref{eq10B}) установлено, если положить в нём $M_1:=1/C_2.$

\medskip
Осталось установить второе неравенство в~(\ref{eq10B}). Прежде
всего, заметим, что если $I=\infty,$ то величина
$\frac{\omega_{n-1}}{I^{n-1}}$ в~(\ref{eq10B}) равна нулю. Тогда и
интеграл $J:=\int\limits_{\widetilde{A}} Q(p)\cdot
\eta_0^n(\widetilde{h}(p, p_0))\ d\widetilde{v}(p)$ равен нулю,
поскольку по лемме~\ref{lem2}
\begin{equation}\label{eq18}
J\leqslant C_1\cdot
\int\limits_{r_1}^{r_2}\int\limits_{\widetilde{S}(p_0,
r)}Q(p)\cdot\eta_0^n(\widetilde{h}(p, p_0))\
\,d\mathcal{H}_{\widetilde{h}}^{n-1}dr=\frac{C_1\cdot\omega_{n-1}}{I^{n-1}}\,.
\end{equation}
Таким образом, второе неравенство в~(\ref{eq10B}) очевидно при
$I=\infty.$ Заметим также, что $I\ne 0,$ поскольку
$q_{x_0}(r)<\infty$ для п.в. $r\in(r_1, r_2)$ ввиду условия $Q\in
L^1(U)$ и по лемме~\ref{lem2}.

\medskip
Предположим, $0<I<\infty.$ Тогда из (\ref{eq8B}) и (\ref{eq9B})
следует, что $q_{p_0}(r)\ne 0$ и $\eta_0(r)\ne\infty$ п.в. в
$(r_1,r_2).$ Полагаем
$$\alpha(r)=rq_{p_0}^{\frac{1}{n-1}}(r)\eta(r)\,,\quad
w(r)=1/rq_{p_0}^{\frac{1}{n-1}}(r)\,.$$
При почти всех $r\in(r_1,r_2)$ будем иметь:
$$\eta(r)=\alpha(r)w(r)$$
и, кроме того, по лемме~\ref{lem2}
\begin{equation}\label{2.3.16}
C\colon=\int\limits_{\widetilde{A}(p_0, r_1, r_2)} Q(p)\cdot
\eta^n(\widetilde{h}(p, p_0))\,d\widetilde{v}(p)
\geqslant\omega_{n-1}\cdot C_2\cdot
\int\limits_{r_1}^{r_2}\alpha^n(r)\cdot w(r)\,dr\,.
\end{equation}
Применяя неравенство Иенсена с весом к выпуклой функции
$\varphi(t)=t^n\,,$ заданной в интервале
$\Omega\,=\,\left(r_1,r_2\right)$ с вероятностной мерой
$\nu(E)=\frac{1}{I}\int\limits_E w(r)dr$ (см.
\cite[Теорема~2.6.2]{Ran}) и делая замену переменных относительно
меры (см. \cite[теорема~15.1, гл.~I]{Sa}), получаем, что
получаем, что
%
%\begin{equation}\label{eq2.3.18}
$$\left(\frac{1}{I}\int\alpha^n(r)w(r)dr\right)^{1/n}\geqslant
\frac{1}{I}\int\alpha(r)w(r)dr=\frac{1}{I}\,,$$
%\end{equation}
%
где мы также использовали тот факт, что $\eta(r)=\alpha(r)\omega(r)$
удовлетворяет соотношению (\ref{eq8AA}). Отсюда,
учитывая~(\ref{eq18}) и (\ref{2.3.16}),  получаем:
$$
C\geqslant\frac{\omega_{n-1}\cdot C_2}{I^{n-1}}\geqslant
\frac{C_2}{C_1}\cdot J\,,
$$
что и доказывает второе неравенство в~(\ref{eq10B}) при
$M_1/M_2:=C_2/C_1,$ или, что то же самое, при $M_2=C_1/C_2^2.$
$\Box$
\end{proof}

\medskip
Заметим, что
\begin{equation}\label{eq19}
\int\limits_{\varepsilon}^{\varepsilon_0}\frac{dr}{\Vert
Q\Vert_{n-1}(r)}=
\int\limits_{\varepsilon}^{\varepsilon_0}\frac{dr}{\left(\int\limits_{D(p_0,r)}
Q^{n-1}(p)\,\mathcal{H}^{n-1}_{\widetilde{h}}\right)^{\,1/(n-1)}}=\frac{1}{\omega^{1/(n-1)}_{n-1}}\cdot\int
\limits_{\varepsilon}^{\varepsilon_0}\frac{dr}{r{\widetilde{q}}^{1/(n-1)}_{p_0}(r)}\,,
\end{equation}
где~$\widetilde{q}_{p_0}(r)$ -- среднее значение функции $Q^{n-1}$
по сфере $\widetilde{S}(p_0, r),$ определённое
соотношением~(\ref{eq32*}).

\medskip
Пусть $D$ -- область в ${\Bbb B}^n/G,$  и пусть $E,$ $F\subset D$ --
произвольные множества. В дальнейшем через $\Gamma(E,F, D)$ мы
обозначаем семейство всех кривых $\gamma:[a,b]\rightarrow D,$
которые соединяют $E$ и $F$ в $D,$ т.е. $\gamma(a)\in
E,\,\gamma(b)\in F$ и $\gamma(t)\in D$ при $t\in(a,\,b).$
Согласно~\cite[разд.~7]{MRSY} отображение $f:D\rightarrow D_*$ будем
называть {\it кольцевым $Q$-отоб\-ра\-же\-нием в точке $p_0\in
\overline{D}$}, если существует $0<r_0<\infty$ такое, что при любых
$0<r_1<r_2<r_0$ выполнено неравенство
\begin{equation}\label{eq1F}
M(f(\Gamma(S_1, S_2, D)))\leqslant \int\limits_{\widetilde{A}}
Q(p)\cdot \eta^n(\widetilde{h}(p, p_0))\,d\widetilde{v}(p)\,,
\end{equation}
где $S_1=\widetilde{S}(p_0, r_1),$ $S_2=\widetilde{S}(p_0, r_2),$
$\widetilde{A}=\widetilde{A}(p_0, r_1, r_2)=\{p\in {\Bbb B}^n/G:
r_1<\widetilde{h}(p, p_0)<r_2\}$ и $\eta:(r_1, r_2)\rightarrow [0,
\infty]$ -- произвольная измеримая по Лебегу функция,
удовлетворяющая условию
\begin{equation}\label{eq*3!!}
\int\limits_{r_1}^{r_2}\eta(r)\,dr\geqslant 1\,.
\end{equation}

\medskip
Имеет место следующее утверждение, аналог которого был доказан ранее
в~\cite[следствие~5]{KRSS} для случая евклидова пространства.

\medskip
\begin{lemma}\label{lem1}
{\sl\, Пусть $n\geqslant 2,$ $G$ -- некоторая группа мёбиусовых
автоморфизмов единичного шара, действующая разрывно в ${\Bbb B}^n$ и
не имеющая в ${\Bbb B}^n$ неподвижных точек. Пусть также группа
$G_*,$ соответствующая пространству ${\Bbb B}^n/G_*,$ состоит из
одного отображения $g(x)=x.$ Предположим, $D$ и $D_{\,*}$~---
области, принадлежащие ${\Bbb B}^n/G$ и ${\Bbb B}^n/G_*$
соответственно, при этом, $\overline{D}$ и $\overline{D_{\,*}}$
являются компактами, кроме того, функция $Q:{\Bbb B}^n/G\rightarrow
(0, \infty)$ интегрируема в некоторой окрестности точки $p_0\in
\overline{D}.$ Тогда каждый нижний $Q$-гомеоморфизм в точке $p_0\in
\overline{D}$ является кольцевым $\widetilde{C}\cdot
Q^{n-1}$-гомеоморфизмом в этой же точке, где $\widetilde{C}>0$ --
некоторая постоянная, зависящая только от окрестности $U.$ }
\end{lemma}

\medskip
\begin{proof}
%Пусть $M_e$ обозначает модуль семейств кривых (поверхностей),
%соответствующий евклидовой метрике и лебеговой мере в~${\Bbb R}^n$
%(см., напр., \cite[пункты~2.2 и 9.2]{MRSY}).
Пусть $U$ -- нормальная окрестность точки $p_0,$
$0<r_0<\widetilde{h}(p_0, \partial U)$ и $\pi$ -- естественная
проекция шара $B_h(0, r_0)$ на шар $\widetilde{B}(p_0, r_0)$ в
$U\subset {\Bbb B}^n/G.$ Заметим, что при $0<r_1<r_2<r_0$ сферы
$S_1=\widetilde{S}(p_0, r_1)$ и $S_2=\widetilde{S}(p_0, r_2)$
являются гомеоморфными образами евклидовых сфер $S^*_1=S(0, r^*_1)$
и $S^*_2=S(0, r^*_2)$ некоторых радиусов $0<r_1^*<r_2^*.$ Положим
$A^*=\{x\in {\Bbb B}^n: r^*_1<|x|<r^*_2\}.$ Рассматривая
вспомогательное отображение~$\widetilde{f}:=f\circ(\pi^{\,-1}|_U)$
шара $B_h(0, r_0)$ в ${\Bbb B}^n$ и применяя равенства Цимера и
Шлыка (см., соответственно,~\cite[теорема~3.13]{Zi}
и~\cite[теорема~1]{Shl}), получаем:
$$
M_e(f(\Gamma(S_1, S_2, D)))=M_e(f(\Gamma(S_1, S_2, D\cap
A)))=M_e(\widetilde{f}(\Gamma(S^*_1, S^*_2, \pi^{\,-1}|_U(A^*\cap
D))))\leqslant
$$
\begin{equation}\label{eq20}
\leqslant M_e(\Gamma(\widetilde{f}(S^*_1), \widetilde{f}(S^*_2),
{\Bbb B}^n))\leqslant \frac{1}{M^{n-1}_e(f(\Sigma_{r_1, r_2}))}\,,
\end{equation}
где~$\Sigma_{r_1, r_2}$ обозначает семейство пересечений сфер
$\widetilde{S}(p_0, r),$ $r\in (r_1, r_2),$ с областью $D.$
Рассуждая аналогично~\cite[пункт~1, замечание~5.2]{Sev$_1$}, можно
показать, что
\begin{equation}\label{eq21}
M_e(f(\Gamma(S_1, S_2, D)))=M_h(f(\Gamma(S_1, S_2,
D)))=M(f(\Gamma(S_1, S_2, D)))\,.
\end{equation}
Покажем, что найдётся постоянная $C>0$ такая, что
\begin{equation}\label{eq22}
C\cdot M_e(f(\Sigma_{r_1, r_2}))\geqslant  M(f(\Sigma_{r_1,
r_2}))\,.
\end{equation}
В самом деле, пусть $\rho\in{\rm adm}_e\,f(\Sigma_{r_1, r_2}).$
Поскольку $\mathcal{H}^{n-1}(E)\leqslant \mathcal{H}_h^{n-1}(E)$ для
произвольного $\mathcal{H}_h^{n-1}$-измеримого множества $E$
согласно неравенству~(\ref{eq1C}), то и $$\int\limits_S \rho^{n-1}\
d\mathcal{\widetilde{A}}=\int\limits_{{\Bbb
B}^n/G}\rho^{n-1}(y)\:N(S,y)\ d\mathcal{H}^{n-1}_hy\geqslant 1$$ для
всех $S\in f(\Sigma_{r_1, r_2}),$ то есть, $\rho\in{\rm
adm}\,f(\Sigma_{r_1, r_2}).$ Заметим также, что
\begin{equation}\label{eq23}M(f(\Sigma_{r_1, r_2}))=\inf\limits_{\rho\in{\rm
adm}\,f(\Sigma_{r_1, r_2})} \int\limits_{{\Bbb
B}^n}\rho^n(p)\,dv(p)\leqslant \int\limits_{{\Bbb
B}^n}\rho^n(p)\,dv(p)\leqslant C\cdot\int\limits_{{\Bbb
B}^n}\rho^n(p)\,dm(p)\,,
\end{equation}
так как по условию область $D_*,$ содержащая семейство
$f(\Sigma_{r_1, r_2}),$ является компактным подмножеством единичного
шара и, значит, множитель $\frac{2^n}{{(1-|x|^2)}^n},$ участвующий в
элементе объёма $dv(p)$ в~(\ref{eq2B}), ограничен сверху в
$\overline{D_*}.$ Переходя в~(\ref{eq23}) к $\inf$ по всем
$\rho\in{\rm adm}_e\,f(\Sigma_{r_1, r_2}),$ получаем требуемое
неравенство~(\ref{eq22}).

\medskip
Объединяя теперь~(\ref{eq20}), (\ref{eq21}) и~(\ref{eq22}),
получаем, что
\begin{equation}\label{eq24}
M(f(\Gamma(S_1, S_2, D)))\leqslant
\frac{C^{n-1}}{M^{n-1}(f(\Sigma_{r_1, r_2}))}\,.
\end{equation}
По лемме~\ref{lem4A} на основании~(\ref{eq24}) получаем:
\begin{equation}\label{eq25}
M(f(\Gamma(S_1, S_2, D)))\leqslant \frac{M\cdot
C^{n-1}}{\left(\int\limits_{\varepsilon}^{\varepsilon_0}
\frac{dr}{\Vert\,Q\Vert_{n-1}(r)}\right)^{n-1}}\,.
\end{equation}
Тогда из~(\ref{eq19}) и~(\ref{eq25}) вытекает, что
\begin{equation}\label{eq26}
M(f(\Gamma(S_1, S_2, D)))\leqslant \frac{M\cdot
C^{n-1}\cdot\omega_{n-1}}{\left(\int
\limits_{r_1}^{r_2}\frac{dr}{r{\widetilde{q}}^{1/(n-1)}_{p_0}(r)}\right)^{n-1}}\,.
\end{equation}
Тогда по предложению~\ref{pr1B} на основании~(\ref{eq26}) получаем,
что
\begin{equation}\label{eq30}
M(f(\Gamma(S_1, S_2, D)))\leqslant M\cdot C^{n-1}\cdot
M_2\cdot\int\limits_{\widetilde{A}} Q(p)\cdot
\eta^n((\widetilde{h}(p, p_0)))\ d\widetilde{v}(p)\,,
\end{equation}
где~$\eta: (r_1,r_2)\rightarrow [0,\infty]$ -- произвольная
измеримая по Лебегу функция, удовлетворяющая условию~(\ref{eq8AA}),
а $M_2$ -- постоянная, соответствующая неравенствам~(\ref{eq10B}).
Лемма доказана с учётом того, что переход к произвольной функции
$\eta$ из условия~(\ref{eq*3!!}) в неравенстве~(\ref{eq30})
достигается путём рассмотрения вспомогательной функции
$\eta^{\,\prime}:=\eta/I(r_1, r_2),$ где $I(r_1,
r_2)/\int\limits_{r_1}^{r_2}\eta(t)\,dt.$~$\Box$
\end{proof}

\medskip
Из теоремы~\ref{thOS4.1} и леммы~\ref{lem1} получаем следующе
важнейшее утверждение.

\medskip
\medskip
\begin{theorem}\label{th1}
{\sl\, Пусть $n\geqslant 3,$ $G$ -- некоторая группа мёбиусовых
автоморфизмов единичного шара, действующая разрывно в ${\Bbb B}^n$ и
не имеющая в ${\Bbb B}^n$ неподвижных точек. Пусть также группа
$G_*,$ соответствующая пространству ${\Bbb B}^n/G_*,$ состоит из
одного отображения $g(x)=x.$ Предположим, $D$ и $D_{\,*}$~---
области, принадлежащие ${\Bbb B}^n/G$ и ${\Bbb B}^n/G_*$
соответственно, при этом, $\overline{D}$ и $\overline{D_{\,*}}$
являются компактами.

\medskip
Пусть также $\varphi\colon(0,\infty)\rightarrow (0,\infty)$~---
неубывающая функция, удовлетворяющая условию
Кальдерона~(\ref{eqOS3.0a}), кроме того, функция $K^{n-1}_O(p, f)$
интегрируема в некоторой окрестности точки $p_0\in \overline{D}.$
Тогда каждый гомеоморфизм $f\colon D\rightarrow {\Bbb B}^n/G_*$
класса $W^{1,\varphi}_{\rm loc},$ имеющий конечное искажение,
является кольцевым $Q$-отображением в каждой точке
$p_0\in\overline{D}$ при $Q(p):=c^{n-1}\cdot  K^{n-1}_O(p, f),$ где
$c>0$ -- некоторая постоянная, а внешняя дилатация $K_O(p, f)$
отображения $f$ в точке $p$ определена соотношением~(\ref{eq1B}). }
\end{theorem}

\medskip
{\bf 6. Приложения. Локальное и граничное поведение отображений.}
Условимся говорить, что граница $\partial D$ области $D$ является
{\it сильно достижимой в точке $p_0\in
\partial D$}, если для каждой окрестности $U$ точки $p_0$ найдётся компакт
$E\subset D,$ окрестность $V\subset U$ этой же точки и число $\delta
>0$ такие, что для любых континуумов $E$ и $F,$ пересекающих как
$\partial U,$ так и $\partial V,$ выполняется неравенство
$M(\Gamma(E, F, D))\geqslant \delta.$ Мы также будем говорить, что
граница $\partial D$  является {\it сильно достижима}, если она
является сильно достижимой в каждой своей точке. Для множества
$E\subset {\Bbb B}^n/G,$ как обычно, обозначим
$$C(f, E)=\{p_*\in{\Bbb B}^n/G_*: \exists\,\, p_k\in D, p\in \partial E:
p_k\rightarrow p, f(p_k)\rightarrow p_*, k\rightarrow\infty\}\,.$$
Следуя~\cite[разд.~2]{IR} либо~\cite[разд.~6.1, гл.~6]{MRSY}, будем
говорить, что функция $\varphi\colon D\rightarrow{\Bbb R}$ имеет
{\it конечное среднее колебание} в точке $p_0\in \overline{D}$,
пишем $\varphi\in FMO(p_0),$ если
%
%
%
%\begin{equation}\label{eq29*!}
%
$$\limsup\limits_{\varepsilon\rightarrow
0}\frac{1}{\widetilde{h}(\widetilde{B}(p_0,
\varepsilon))}\int\limits_{\widetilde{B}(p_0,\,\varepsilon)}
|{\varphi}(p)-\overline{\varphi}_{\varepsilon}|\
d\widetilde{h}(p)<\infty\,,$$
%
%\end{equation}
%
где
$\overline{{\varphi}}_{\varepsilon}=\frac{1}
{\widetilde{h}(\widetilde{B}(p_0,
\varepsilon))}\int\limits_{\widetilde{B}(p_0, \varepsilon)}
{\varphi}(p) \,d\widetilde{h}(p).$ Справедливо следующее
утверждение.

\medskip
\begin{theorem}\label{th3}{\sl\, Пусть $n\geqslant 3,$ $G$ -- некоторая группа мёбиусовых
автоморфизмов единичного шара, действующая разрывно в ${\Bbb B}^n$ и
не имеющая в ${\Bbb B}^n$ неподвижных точек, при этом, группа $G_*$
состоит из одного отображения $g(x)=x.$ Пусть также $D$ и
$D_{\,*}$~--- области, принадлежащие ${\Bbb B}^n/G$ и ${\Bbb
B}^n/G_*$ соответственно, при этом, $\overline{D}$ и
$\overline{D_{\,*}}$ являются компактами, кроме того, область $D$
локально линейно связна в точке $b\in \partial D,$ а $\partial
D^{\,\prime}$ сильно достижима хотя бы в одной из точек $y\in C(f,
b).$

\medskip
Предположим, $\varphi\colon(0,\infty)\rightarrow (0,\infty)$~---
неубывающая функция, удовлетворяющая условию
Кальдерона~(\ref{eqOS3.0a}), кроме того пусть заданная функция
$Q:{\Bbb B}^n/G\rightarrow (0, \infty)$ интегрируема в некоторой
окрестности точки $p_0\in \overline{D}.$ Тогда, если $f$ --
гомеоморфизм области $D$ на $D_*$ класса $W^{1,\varphi}_{\rm loc},$
такой что $K^{n-1}_O(p, f)\leqslant Q(p)$ при почти всех $p\in D,$
при этом, $Q\in FMO(b),$ то $C(f, b)=\{y\}.$ }
\end{theorem}

\medskip
\begin{proof}
Пусть $\eta:(r_1, r_2)\rightarrow [0, \infty]$ -- произвольная
измеримая по Лебегу функция, удовлетворяющая условию
$\int\limits_{r_1}^{r_2}\eta(t)\,dt\geqslant 1.$ По
теореме~\ref{th1} отображение $f$ удовлетворяет соотношению
\begin{equation}\label{eq15A}
M(f(\Gamma(C_1, C_0, A)))\leqslant\int\limits_{A\cap D}c^*\cdot
Q(p)\cdot\eta^n(\widetilde{h}(p, p_0))\,d\widetilde{h}(p)
\end{equation}
для любых двух континуумов $C_0\subset \overline{\widetilde{B}(b,
r_1)},$ $C_1\subset {\Bbb B}^n/G\setminus \widetilde{B}(b, r_2),$
некотором $0<r_0<\infty$ и всех $0<r_1<r_2<r_0,$ где $c^*>0$ --
некоторая постоянная и
$$\widetilde{A}=\widetilde{A}(b, r_1, r_2)=\{p\in {\Bbb B}^n/G: r_1<d(p, p_0)<r_2\}\,.$$
Из~\cite[лемма~2.4]{Sev$_1$} вытекает, что
фактор-пространства~${\Bbb B}^n/G$ являются $n$-регулярными по
Альфорсу и, значит, имеют хаусдорфову размерность~$n$ (см., напр.,
\cite[пункт~8.7]{He}). В таком случае, необходимое заключение
вытекает из~\cite[теорема~5]{Sev$_2$}.~$\Box$
\end{proof}

\medskip
\begin{theorem}\label{th4}{\sl\, Теорема~\ref{th3} остаётся в силе, если в условиях
этой теоремы предположение $Q\in FMO(b)$ заменить предположением
\begin{equation}\label{eq31}
\int\limits_{\varepsilon}^{\varepsilon_0}\frac{dr}{rq^{1/(n-1)}_{p_0}(r)}\rightarrow
\infty\,,\quad \varepsilon\rightarrow 0\,,
\end{equation}
где $0<\varepsilon_0<\sup\limits_{p\in D}\widetilde{h}(p, p_0)$
 -- некоторое число и $q_{p_0}(r)$ определено в~(\ref{eq32*}).
}
\end{theorem}

\medskip
\begin{proof}
При каждом фиксированном $\varepsilon<\varepsilon_0,$
$\varepsilon>0,$ рассмотрим функцию
$I(\varepsilon, \varepsilon_0)=\int\limits
_{\varepsilon}^{\varepsilon_0}\psi(t)\,dt,$
где
%\begin{equation}\label{eq1*****A}
$$\psi(t)\quad=\quad \left \{\begin{array}{rr}
1/(tq^{\frac{1}{n-1}}_{p_0}(t))\ , & \ t\in (\varepsilon,
\varepsilon_0)\ ,
\\ 0\ ,  &  \ t\notin (\varepsilon,
\varepsilon_0)\ ,
\end{array} \right.$$
%\end{equation}
%
Здесь, как обычно, полагаем $a/\infty = 0$ при $a\ne\infty,$
$a/0=\infty $ при $a>0$ и $0\cdot\infty =0,$ см., напр.,
\cite[с.~18, $\S\,3,$ гл.~I]{Sa}). Заметим, что $I(\varepsilon,
\varepsilon_0)<\infty$ при всех $\varepsilon\in (0, \varepsilon_0),$
так как $K^{n-1}_O(p, f)\leqslant Q(p)$ и, значит, $Q(p)\geqslant 1$
почти всюду (откуда также вытекает, что и $q_{p_0}(t)\geqslant 1$
при почти всех $t\in (0, \varepsilon_0)).$

\medskip
Применяя аналог теоремы Фубини (лемму~\ref{lem2}), получаем, что
$$\int\limits_{\varepsilon<\widetilde{h}(p, p_0)<\varepsilon_0} Q(p)\cdot\psi^n(\widetilde{h}(p, p_0))\
d\widetilde{v}(p)\leqslant C_1\cdot
\int\limits_{\varepsilon}^{\varepsilon_0}
\int\limits_{\widetilde{S}(p_0, r)} Q(p)\cdot\psi^n(\widetilde{h}(p,
p_0))\,d\mathcal{H}^{n-1}_{\widetilde{h}}\,dr=$$
\begin{equation}\label{eq2.3.3}
=C_1\cdot\int\limits_{\varepsilon}^{\varepsilon_0}\omega_{n-1}r^{n-1}q_{p_0}(r)\psi^n(r)\,dr=
 \omega_{n-1}\cdot C_1\cdot
\int\limits_{\varepsilon}^{\varepsilon_0}\frac{dr}{rq_{p_0}^{\frac{1}{n-1}}(r)}\,,
\end{equation}
причём
$\int\limits_{\varepsilon}^{\varepsilon_0}\frac{dr}{tq_{p_0}^{\frac{1}{n-1}}(r)}
=o\left(I^n(\varepsilon,\varepsilon_0)\right)$
ввиду~(\ref{eq31}). Оставшаяся часть утверждения вытекает из
рассуждений, аналогичных доказательству теоремы~\ref{th3}, а также
ввиду~\cite[лемма~7]{Sev$_2$}.
\end{proof}~$\Box$

\medskip
Напомним, что семейство $\frak{F}$ отображений $f\colon X\rightarrow
{X}^{\,\prime}$ называется {\it равностепенно непрерывным в точке}
$x_0 \in X,$ если для любого $\varepsilon>0$ найдётся такое
$\delta>0$, что ${d}^{\,\prime}
\left(f(x),f(x_0)\right)<\varepsilon$ для всех таких $x,$ что
$d(x,x_0)<\delta$ и для всех $f\in \frak{F}.$ Говорят, что
$\frak{F}$ {\it равностепенно непрерывно}, если $\frak{F}$
равностепенно непрерывно в каждой  точке $x_0\in X.$ В качестве
одного из возможных приложений теоремы~\ref{th1} приведём также
следующее утверждение.

\medskip
\begin{theorem}\label{theor4*!} {\sl\, Пусть $n\geqslant 3,$ $G$ -- некоторая группа мёбиусовых
автоморфизмов единичного шара, действующая разрывно в ${\Bbb B}^n$ и
не имеющая в ${\Bbb B}^n$ неподвижных точек, при этом, группа $G_*$
состоит из одного отображения $g(x)=x.$ Пусть также $D$ -- область в
${\Bbb B}^n/G,$ такая, что $\overline{D}$ является компактом.
Предположим, $\varphi\colon(0,\infty)\rightarrow (0,\infty)$~---
неубывающая функция, удовлетворяющая условию
Кальдерона~(\ref{eqOS3.0a}), кроме того пусть заданная функция
$Q:{\Bbb B}^n/G\rightarrow (0, \infty)$ интегрируема в некоторой
окрестности точки $p_0\in \overline{D}.$

Пусть также $p_0\in D,$ $0<R<1,$ $B_R=\{x\in {\Bbb B}^n: |x|<R\},$ и
$Q:D\rightarrow [0, \infty]$ -- измеримая относительно меры
$\widetilde{v}$ функция.

Обозначим через $\frak{R}_{\varphi, Q, B_R, \delta}(D)$ семейство
всех гомеоморфизмов $f:D\rightarrow B_R$ класса~$W_{\rm loc}^{1,
\varphi}$ с конечным искажением для которых: 1) найдётся континуум
$K_f\subset B_R \setminus f(D),$ удовлетворяющий условию
$\sup\limits_{x, y\in K_f} h(x, y)\geqslant \delta>0;$ 2)
$K^{n-1}_O(p, f)\leqslant Q(p)$ для почти всех $p\in D.$ Тогда, если
$Q\in FMO(p_0),$ либо в точке $p_0$ выполнено условие
вида~(\ref{eq31}), то семейство отображений $\frak{R}_{\varphi, Q,
B_R, \delta}(D)$ является равностепенно непрерывным в точке $p_0\in
D.$ }
\end{theorem}

\medskip
{\it Доказательство} теоремы~\ref{theor4*!}, с учётом замечаний,
сделанных при доказательстве теорем~\ref{th3} и \ref{th4}, вытекает
из теоремы~\ref{th1} и \cite[теорема~1]{Sev$_2$}. Здесь следует
также учесть, что шар $B_R$ является регулярным по Альфорсу
пространством с $(1; n)$-неравенством Пуанкаре (см., напр.,
\cite[лемма~2.4]{Sev$_1$}, \cite[пример~8.24,
предложение~8.19]{He}).~$\Box$

\medskip
{\bf 7. Пример.} В качестве простой иллюстрации к
теореме~\ref{theor4*!}, рассмотрим следующее семейство отображений.
Опираясь на гомеоморфизм $h(x)=\frac{x}{|x|}\log\frac{1}{|x|},$
положим
$$
h_m(x)=\left\{
\begin{array}{rr}
\frac{x}{|(m-1)/m|}\cdot\log\left(\frac{e}{(m-1)/m}\right), & x\in B(0, (m-1)/m),\\
\frac{x}{|x|}\cdot\log\left(\frac{e}{|x|}\right), &  x\in {\Bbb
B}^n\setminus B(0, (m-1)/m)
\end{array}
\right.\,,
$$
где ${\Bbb B}^n=\{x\in {\Bbb R}^n: |x|<1\}.$

\medskip Предположим, $G$ -- разрывная группа мёбиусовых отображений
единичного шара ${\Bbb B}^n,$ $n\geqslant 2,$ на себя, не имеющая
неподвижных точек в ${\Bbb B}^n.$ Пусть $\pi:{\Bbb B}^n\rightarrow
{\Bbb D}/G$ -- естественная проекция ${\Bbb B}^n$ на ${\Bbb B}^n/G,$
и пусть точка $p_0\in {\Bbb B}^n/G$ такова, что $\pi(0)=p_0.$
Рассмотрим $r_0>0$ -- радиус шара с центром в точке~$p_0,$ целиком
лежащий в некоторой нормальной окрестности $U$ точки $p_0.$ Тогда по
определению естественной проекции $\pi,$  а также определению
гиперболической метрики $h$ и метрики~$\widetilde{h}$ в~(\ref{eq2})
мы имеем, что $\pi(B(0, r^{\,\prime}_0))=\widetilde{B}(p_0, r_0),$
где $r^{\,\prime}_0:=(e^{r_0}-1)/(e^{r_0}+1).$ В таком случае,
семейство отображений
$$
\widetilde{g}_m(y)=\left\{
\begin{array}{rr}
\frac{y}{(m-1)/m}\cdot\log\left(\frac{e}{(m-1)/m}\right),
& y\in B(0, (r^{\,\prime}_0(m-1))/m),\\
\frac{r^{\,\prime}_0y}{|y|}\cdot\log\left(\frac{er^{\,\prime}_0}{|y|}\right),
& y\in B(0, r^{\,\prime}_0)\setminus B(0, (r^{\,\prime}_0(m-1))/m)
\end{array}
\right.\,,
$$
является семейством автоморфизмов шара $B(0, r^{\,\prime}_0),$ при
этом, $\widetilde{g}_m|_{S(0, r^{\,\prime}_0)}=y.$

\medskip
Заметим, что $\widetilde{g}_m\in ACL({\Bbb B}^n).$ В самом деле,
отображения
$$\widetilde{g^{(1)}}_m(y)=\frac{y}{|(m-1)/m|\log\left(\frac{e}{(m-1)/m}\right)},
\quad m=1,2,\ldots,\,$$
являются отображениями класса $C^1,$ скажем, в шаре $B(0,
r^{\,\prime}_0(m-1)/m+\varepsilon)$ при малых $\varepsilon>0,$ а
отображения
$\widetilde{g^{(2)}}_m(y)=\frac{y}{|y|\log\left(\frac{e}{|y|}\right)}$
~--- отображениями класса $C^1,$ скажем, в кольце $B(0,
r^{\,\prime}_0\setminus B(0, r^{\,\prime}_0(m-1)/m)-\varepsilon).$
Отсюда вытекает, что гомеоморфизмы $\widetilde{g}_m$ являются
липшицевыми в $B(0, r^{\,\prime}_0)$ и, значит, $\widetilde{g}_m\in
ACL(B(0, r^{\,\prime}_0),$ см., напр., \cite[разд.~5]{Va}.

\medskip
Далее, согласно~\cite[пример~7, пункт~1.3.6]{Sev$_3$} (см. также
рассуждения, приведённые при
рассмотрении~\cite[предложение~6.3]{MRSY})
$$
\Vert \widetilde{g}^{\,\prime}_m(y)\Vert =\left\{
\begin{array}{rr}
\frac{1}{(m-1)/m}\cdot \log\left(\frac{e}{(m-1)/m}\right), & y\in B(0, r^{\,\prime}_0(m-1)/m),\\
\frac{r^{\,\prime}_0}{|y|}\cdot\log\frac{er^{\,\prime}_0}{|y|}, &
y\in B(0, r^{\,\prime}_0)\setminus B(0, r^{\,\prime}_0(m-1)/m)
\end{array}
\right.\,,$$
$$|J(y, \widetilde{g}_m)|=\left\{
\begin{array}{rr}
\frac{1}{(m-1)/m}\cdot\left(\log\left(\frac{e}{(m-1)/m}\right)\right)^n, & y\in B(0, r^{\,\prime}_0(m-1)/m),\\
\frac{(r_0{^{\,\prime})^n}}{|y|^n}\cdot\log^{n-1}\frac{er_0^{\,\prime}}{|y|},
& y\in B(0, r^{\,\prime}_0)\setminus B(0, r^{\,\prime}_0(m-1)/m)
\end{array}
\right.\,.
$$
Отсюда вытекает, что $\Vert
\widetilde{g}^{\,\prime}_m(y)\Vert\leqslant c_m$ при некоторой
постоянной $c_m>0$ и, значит, $\widetilde{g}^{\,\prime}_m\in W_{\rm
loc}^{1, 1}(B(0, r^{\,\prime}_0),$ так как
$\widetilde{g}^{\,\prime}_m\in ACL(B(0, r^{\,\prime}_0)$
(см.~\cite[теорема~2, пункт~1.1.3]{Ma}). Заметим, что $|\nabla
\widetilde{g}_m(y)|\leqslant n^{1/2}\cdot\Vert
\widetilde{g}^{\,\prime}_m(y)\Vert$ при почти всех $y\in B(0,
r^{\,\prime}_0).$ Пусть $\varphi\colon(0,\infty)\rightarrow
(0,\infty)$~--- неубывающая функция, удовлетворяющая условию
Кальдерона~(\ref{eqOS3.0a}). Тогда ввиду неубывания функции
$\varphi$
$$\int\limits_{B(0, r^{\,\prime}_0)}\varphi(|\nabla \widetilde{g}_m(y)|)\,dv(y)
\leqslant\varphi(n^{1/2}c_m)\cdot v(B(0, r^{\,\prime}_0))<\infty\,,$$
т.е., $\widetilde{g}_m\in W^{1, \varphi}(B(0, r^{\,\prime}_0)).$
Заметим, что отображения $\widetilde{g_m}$ имеют конечное искажение,
поскольку их якобиан почти всюду не равен нулю; кроме того,
$K_O^{n-1}(y, \widetilde{g_m})\leqslant Q(y),$ где
$Q(y)=\log^{n-1}\frac{er_0^{\,\prime}}{|y|}.$

\medskip
С другой стороны, по построению, существует и непрерывно обратное
отображение $\varphi=\pi^{\,-1}$ шара $\widetilde{B}(p_0, r_0)$ на
евклидов шар $B(0, r^{\,\prime}_0).$ В таком случае, положим
$$g_m(p)=(\widetilde{g}_m\circ \varphi)(p), \qquad p\in \widetilde{B}(p_0, r_0)\,.$$
Заметим, что отображения $g_m$ являются гомеоморфизмами
$\widetilde{B}(p_0, r_0)$ на евклидов шар $B(0, r^{\,\prime}_0)$
класса $W^{1, \varphi}(B(0, r^{\,\prime}_0)).$ Кроме того, имеем,
что $K_O^{n-1}(p, g_m)\leqslant \widetilde{Q}(p),$ где
$\widetilde{Q}(p)=\log^{n-1}\frac{er_0^{\,\prime}}{|\varphi(p)|}.$
См. рисунок~\ref{fig3} для иллюстрации.
\begin{figure}[h]
\centerline{\includegraphics[scale=0.42]{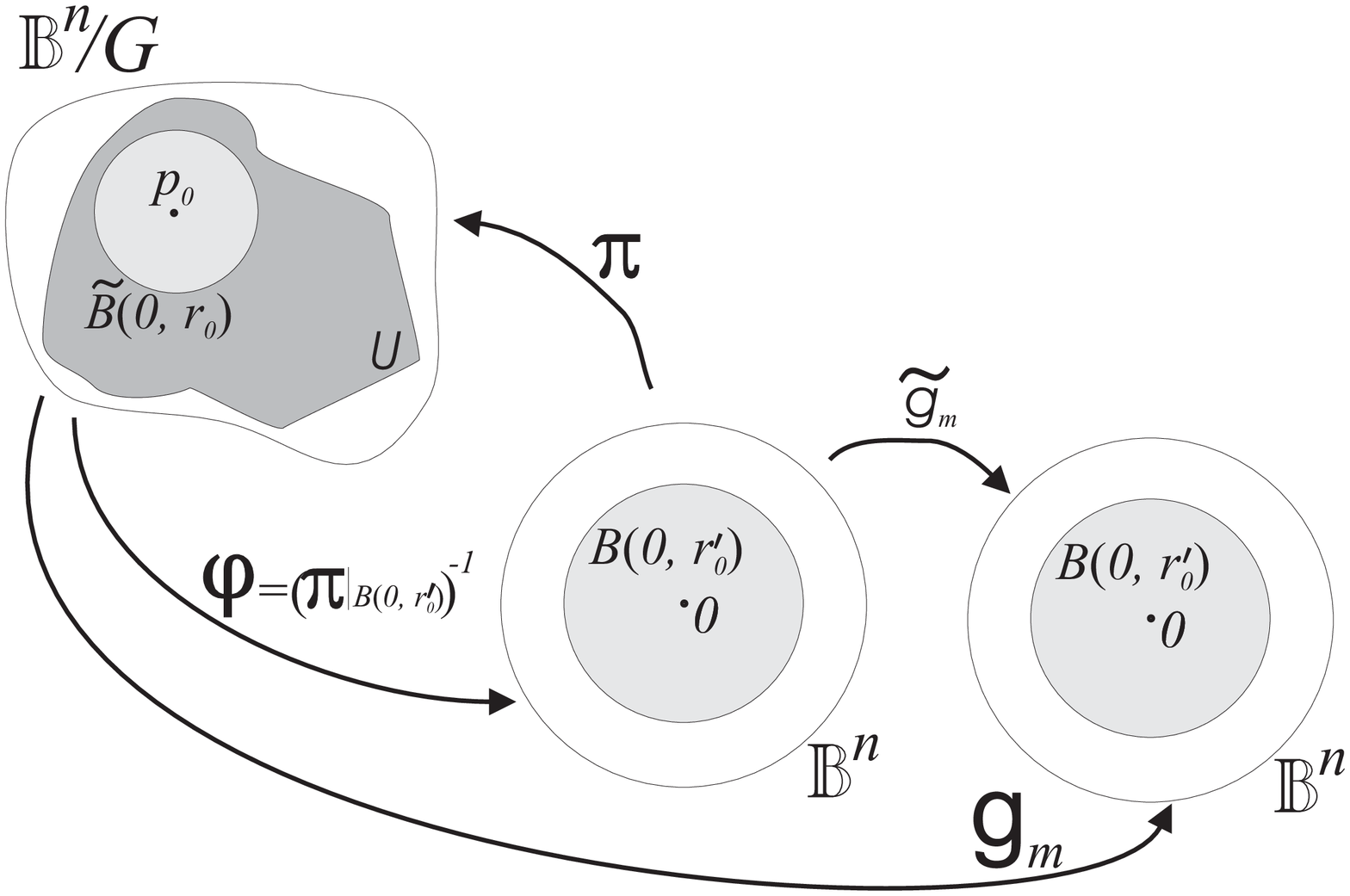}}
\caption{Конструкция, использованная при построении отображений
$g_m$ }\label{fig3}
\end{figure}
Согласно неравенств~(\ref{eq3B}) при некоторой положительной
постоянной $c*_1:=1/c_1\geqslant 1,$ зависящей только от~$r_0,$
будем иметь:
$$\log^{n-1}\frac{er_0^{\,\prime}}{|\varphi(p)|}\leqslant \log^{n-1}\left(\frac{c^*_1er_0^{\,\prime}}{h(\varphi(p),
0)}\right)=$$
$$=\log^{n-1}\left(\frac{C}{h(\varphi(p),
0)}\right) :=Q_1(p)\,,$$
где $C:=er_0^{\,\prime}c^*_1.$ Ввиду соотношений~(\ref{eq1C})
$$q^*_{p_0}(r)=\frac{1}{\omega_{n-1}r^{n-1}}
\int\limits_{\widetilde{S}(p_0,
r)}Q_1\,\mathcal{H}^{n-1}_{\widetilde{h}}=
\frac{1}{\omega_{n-1}r^{n-1}}\cdot
\log^{n-1}\left(\frac{C}{r}\right)\cdot \mathcal{H}^{n-1}_h(S_h(0,
r))\leqslant
$$
$$\leqslant \frac{1}{c^*_1\omega_{n-1}r^{n-1}}\cdot
\log^{n-1}\left(\frac{C}{r}\right)\cdot \mathcal{H}^{n-1}(S_h(0,
r))=\frac{1}{c^*_1\omega_{n-1}r^{n-1}}\cdot
\log^{n-1}\left(\frac{C}{r}\right)\cdot
\left(\frac{e^r-1}{e^r+1}\right)^{n-1}\leqslant $$
\begin{equation}\label{eq33}
\leqslant C_1\cdot \log^{n-1}\left(\frac{C}{r}\right)
\end{equation}
при некоторой постоянной $C_1=C_1(r_0)>0,$ поскольку по правилу
Лопиталя $\lim\limits_{r\rightarrow
+0}\frac{\left(\frac{e^r-1}{e^r+1}\right)}{r}=\frac{1}{2}.$
Выберем $0<\varepsilon_0<r^{\,\prime}_0.$ Тогда ввиду
соотношений~(\ref{eq33}) имеет место соотношение~(\ref{eq31}).

\medskip
Непосредственным образом можно проверить, что семейство отображений
$g_m$ является равностепенно непрерывным в точке~$p_0.$ Последнее
заключение вытекает также из теоремы~\ref{theor4*!}. Заметим, что
все условия этой теоремы выполнены, кроме того, что семейство
отображений $g_m$ действует в некоторый не принимает значения из
некоторого континуума $K_m,$ принадлежащего этому шару и имеющего
диаметр не меньший $\delta>0,$ $m=1,2,\ldots .$

\medskip
Для того, чтобы это последнее условие выполнялось, можно рассмотреть
сужение $g_m|_{\widetilde{B}(p_0, r^*_0)}$ данного семейства
отображений $g_m$ на меньший шар $\widetilde{B}(p_0, r^*_0),$
$r^*_0=r_0/2.$ Тогда семейство $g_m|_{\widetilde{B}(p_0, r^*_0)}$
удовлетворяет всем условиям теоремы~\ref{theor4*!}, в частности,
отображения этого семейства не принимают значения некоторого
фиксированного невырожденного континуума~$K\subset {\Bbb B}^n.$

КОНТАКТНАЯ ИНФОРМАЦИЯ

\medskip
\noindent{{\bf Евгений Александрович Севостьянов} \\
{\bf 1.} Житомирский государственный университет им.\ И.~Франко\\
кафедра математического анализа, ул. Большая Бердичевская, 40 \\
г.~Житомир, Украина, 10 008 \\
{\bf 2.} Институт прикладной математики и механики
НАН Украины, \\
отдел теории функций, ул.~Добровольского, 1 \\
г.~Славянск, Украина, 84 100\\
e-mail: esevostyanov2009@gmail.com}

\end{document}